\documentclass[12pt,reqno]{amsart}
\usepackage{amsmath,amsfonts,amssymb,amsthm,amsbsy,latexsym,cite,amscd}
\usepackage[cp1251]{inputenc}
\usepackage[T2A]{fontenc}
\usepackage[english]{babel}
\usepackage[colorlinks,breaklinks,linkcolor=blue,citecolor=blue,menucolor=blue]{hyperref}

\multlinegap=\parindent

\begin{document}

\newtheorem{etheorem}{Theorem}
\newtheorem{eproposition}{Proposition}[section]
\newtheorem{ecorollary}{Corollary}

\theoremstyle{definition}

\newtheorem*{eexample}{Example}

\renewcommand{\thefootnote}{}

\title[Certain residual properties of~HNN-ex\-ten\-sions]{Certain residual properties of~HNN-ex\-ten\-sions\\ with~normal associated subgroups}

\author{E.~V.~Sokolov \lowercase{and}~E.~A.~Tumanova}
\address{Ivanovo State University, Russia}
\email{ev-sokolov@yandex.ru, helenfog@bk.ru}

\begin{abstract}
Let $\mathbb{E}$ be~the~HNN-ex\-ten\-sion of~a~group~$B$ with~subgroups~$H$ and~$K$ associated according to~an~isomorphism $\varphi\colon H \to K$. Suppose that $H$ and~$K$ are~normal in~$B$ and~$(H \cap K)\varphi = H \cap K$. Under~these assumptions, we~prove necessary and~sufficient conditions for~$\mathbb{E}$ to~be~residually a~$\mathcal{C}$\nobreakdash-group, where $\mathcal{C}$ is~a~class of~groups closed under taking subgroups, quotient groups, and~unrestricted wreath products. Among other things, these conditions give new facts on~the~residual finiteness and~the~residual $p$\nobreakdash-fi\-nite\-ness of~the~group~$\mathbb{E}$.
\end{abstract}


\thanks{The~study was~supported by~the~Russian Science Foundation grant No.~24-21-00307,\\ \url{http://rscf.ru/en/project/24-21-00307/}}

\maketitle

\section{Introduction}\label{es01}

Let $\mathcal{C}$ be~a~class of~groups. Following~\cite{Hall1954PLMS}, we~say that a~group~$X$ is~\emph{residually a~$\mathcal{C}$\nobreakdash-group} if any of~its non-trivi\-al elements is~mapped to~a~non-trivi\-al element by~a~suitable homomorphism of~$X$ onto~a~$\mathcal{C}$\nobreakdash-group (i.e.,~a~group from~the~class~$\mathcal{C}$). Recall that, if~$\mathcal{C}$ is~the~class of~all finite groups (or~all solvable groups, or~finite $p$\nobreakdash-groups, where $p$ is~a~prime), then a~residually $\mathcal{C}$\nobreakdash-group is~also referred to~as~a~\emph{residually finite} (respectively, \emph{residually solvable}, \emph{residually $p$\nobreakdash-fi\-nite}) group. We~therefore use the~term ``\emph{residual $\mathcal{C}$\nobreakdash-ness}'' along with~the~well-known notions of~residual finiteness, residual $p$\nobreakdash-fi\-nite\-ness, and~residual solvability. Let~us clarify that the~residual $\mathcal{C}$\nobreakdash-ness of~a~group~$X$ is~the~same as~the~property of~$X$ to~be~residually a~$\mathcal{C}$\nobreakdash-group. In~Section~\ref{es01}, for~brevity, this term will also assume that $\mathcal{C}$ is~a~root class of~groups.

According to~\cite{Gruenberg1957PLMS, AzarovTieudjo2002PIvSU}, a~class of~groups~$\mathcal{C}$ is~called a~\emph{root class} if it contains non-trivi\-al groups, is~closed under taking subgroups, and~satisfies any of~the~following conditions, the~equivalence of~which is~proved in~\cite{Sokolov2015CA}:

1)\hspace{1ex}for~every group~$X$ and~for~every subnormal series $1 \leqslant Z \leqslant Y \leqslant X$ whose factors $X/Y$ and~$Y/Z$ belong to~$\mathcal{C}$, there exists a~normal subgroup~$T$ of~$X$ such that $X/T \in \mathcal{C}$ and~$T \leqslant Z$ (\emph{Gruenberg's condition});

2)\hspace{1ex}the~class~$\mathcal{C}$ is~closed under taking unrestricted wreath products;

3)\hspace{1ex}the~class~$\mathcal{C}$ is~closed under taking extensions and,~together with~any two groups~$X$ and~$Y$, contains the~unrestricted direct product $\prod_{y \in Y}X_{y}$, where $X_{y}$ is~an~isomorphic copy of~$X$ for~each $y \in Y$.

Examples of~root classes are~the~classes of~all finite groups, finite $p$\nobreakdash-groups (where $p$ is~a~prime), periodic $\mathfrak{P}$\nobreakdash-groups of~finite exponent (where $\mathfrak{P}$ is~a~non-empty set of~primes), all solvable groups, and~all torsion-free groups. It~is~also easy to~see that the~intersection of~a~family of~root classes is~again a~root class if it contains a~non-trivi\-al group. The~use of~the~concept of~a~root class turned~out to~be~very productive in~studying the~residual properties of~free constructions of~groups: free and~tree products, HNN-ex\-ten\-sions, fundamental groups of~graphs of~group,~etc. Clearly, it~enables~us to~prove several statements at~once instead of~just one. But~what is~more important, the~facts on~residual $\mathcal{C}$\nobreakdash-ness, where $\mathcal{C}$ is~an~\emph{arbitrary} root class of~groups, and~the~methods for~finding them are~well compatible with~each other. This allows one to~easily move from~one free construction to~another and~quickly complicate the~groups under consideration (see, e.g.,~\cite{Tumanova2017SMJ, SokolovTumanova2017MN, SokolovTumanova2019AL, SokolovTumanova2020SMJ, Sokolov2021SMJ2, Sokolov2021JA, SokolovTumanova2023SMJ}).

When studying the~residual $\mathcal{C}$\nobreakdash-ness of~a~group-the\-o\-ret\-ic construction, the~main question is~whether the~construction inherits this property from~the~groups that compose~it. As~a~rule, this question can only be~answered by~imposing various restrictions on~the~above-men\-tioned groups and~their subgroups. The~goal of~this paper is~to~find necessary and~sufficient conditions for~the~residual $\mathcal{C}$\nobreakdash-ness of~an~HNN-ex\-ten\-sion whose associated subgroups are~normal in~the~base group (here and~below we follow~\cite{LyndonSchupp1977} in~the~use of~terms related to~HNN-ex\-ten\-sions).

Considering the~known results on~the~residual $\mathcal{C}$\nobreakdash-ness of~HNN-ex\-ten\-sions, one can observe that most of~them concern the~case of~residual finiteness~\cite{Meskin1972TAMS, BaumslagBTretkoff1978CA, Shirvani1985AM, RaptisVarsos1992JAMS, KimTang1999CMB, Moldavanskii2002BIvSU, BorisovSapir2005IM, WongWong2012MS, Azarov2013SMJ2, WongWong2014AC, Azarov2014MN, Logan2018CA, AsriWongWong2019MS}. The~papers~\cite{RaptisVarsos1991JPAA, Chatzidakis1994IJM, Moldavanskii2000BIvSU, Moldavanskii2003BIvSU, Moldavanskii2006BIvSU, AschenbrennerFriedl2011JPAA} give some facts on~the~residual $p$\nobreakdash-fi\-nite\-ness of~this construction. In~\cite{Tieudjo2005JAMPA, Tumanova2014MAIS, Goltsov2015MN, Tumanova2017SMJ, SokolovTumanova2017MN, SokolovTumanova2019AL, Tumanova2019SMJ, Tumanova2020IVM, Sokolov2022CA, Sokolov2023JGT, SokolovTumanova2023SMJ}, the~residual $\mathcal{C}$\nobreakdash-ness of~HNN-ex\-ten\-sions is~studied provided $\mathcal{C}$ is~an~arbitrary root class of~groups, which possibly satisfies some additional restrictions. It~should be~noted that, at~the~time of~writing these lines, the~results of~the~last listed papers generalize all known facts on~the~residual solvability and~residual $\mathfrak{P}$\nobreakdash-fi\-nite\-ness of~HNN-ex\-ten\-sions, where $\mathfrak{P}$ is~a~non-empty set of~primes, as~above.

The~main method for~studying the~residual $\mathcal{C}$\nobreakdash-ness of~free constructions of~groups is~the~so-called ``filtration approach''. It~was originally proposed in~\cite{BaumslagG1963TAMS1} to~study the~residual finiteness of~the~free product of~two groups with~an~amalgamated subgroup. After a~number of~generalizations and~adaptations~\cite{BaumslagBTretkoff1978CA, Hempel1987AMS, Shirvani1987JPAA, Loginova1999SMJ, Moldavanskii2000BIvSU, Tumanova2014MN}, this approach was~extended in~\cite{Sokolov2021SMJ1} to~the~case of~an~arbitrary root class~$\mathcal{C}$ and~the~fundamental group of~an~arbitrary graph of~groups. The~method includes two steps and, when applied to~HNN-ex\-ten\-sions, can be~described as~follows.

The~first step is~to~find conditions for~an~HNN-ex\-ten\-sion~$\mathbb{E}$ to~have a~homomorphism onto~a~$\mathcal{C}$\nobreakdash-group that acts injectively on~the~base group. The~existence of~such a~homomorphism is~sometimes equivalent to~the~residual $\mathcal{C}$\nobreakdash-ness of~$\mathbb{E}$. For~example, this equivalence holds if $\mathcal{C}$ consists of~finite groups. But~in~general this is~not the~case~\cite{SokolovTumanova2020LJM, Sokolov2023LJM}.

We~call the~assertions proved at~this step the~\emph{first-lev\-el results}. It~should be~noted that there are~no general approaches to~finding them and~each new fact of~this type is~very valuable. Examples of~such assertions are~the~theorem on~the~residual finiteness of~an~HNN-ex\-ten\-sion of~a~finite group~\cite{BaumslagBTretkoff1978CA} and~the~criteria for~the~residual $p$\nobreakdash-fi\-nite\-ness of~the~same construction~\cite{RaptisVarsos1991JPAA, Chatzidakis1994IJM, Moldavanskii2000BIvSU, AschenbrennerFriedl2011JPAA}. When $\mathcal{C}$ is~an~arbitrary root class of~groups closed under taking quotient groups, first-lev\-el results are~found~for
\begin{list}{}{\topsep=0pt\itemsep=0pt\labelsep=1ex\labelwidth=1.1ex\leftmargin=\parindent\addtolength{\leftmargin}{\labelsep}\addtolength{\leftmargin}{\labelwidth}}
\item[--]an~HNN-ex\-ten\-sion with~coinciding associated subgroups that are~normal in~the~base group~\cite{Tumanova2014MAIS};
\item[--]a~number of~HNN-ex\-ten\-sions in~which at~least one of~the~associated subgroups lies in~the~center of~the~base group~\cite{SokolovTumanova2019AL, Sokolov2022CA, SokolovTumanova2023SMJ}.
\end{list}
In~this paper, we~supplement this list and~prove a~criterion for~the~existence of~a~homomorphism with~the~properties described above provided the~associated subgroups of~an~HNN-ex\-ten\-sion~$\mathbb{E}$ are~normal in~the~base group, while their intersection is~normal in~$\mathbb{E}$ (see Theorem~\ref{et01} below).

Omitting technical details, we~can say that the~second step of~the~method consists in~finding a~sufficiently large number of~homomorphisms mapping the~HNN-ex\-ten\-sion~$\mathbb{E}$ onto~HNN-ex\-ten\-sions satisfying the~conditions of~the~previously gotten first-lev\-el results. This allows~us to~prove the~residual $\mathcal{C}$\nobreakdash-ness of~$\mathbb{E}$ when the~base group does~not necessarily belong to~the~class~$\mathcal{C}$. We~call the~sufficient conditions of~the~residual $\mathcal{C}$\nobreakdash-ness found at~this step the~\emph{se\-cond-lev\-el results}. As~a~rule, they can be~proved only under restrictions stronger than those at~step~1. This is~illustrated, in~particular, by~many years of~studying the~property of~residual finiteness, during which a~universal criterion for~the~residual finiteness of~an~HNN-ex\-ten\-sion of~an~arbitrary residually finite group was~never found. In~the~present paper, the~se\-cond-lev\-el results, Theorems~\ref{et04}\nobreakdash--\ref{et06}, are~also proved under certain assumptions supplementing the~conditions of~the~criterion given by~Theorem~\ref{et01}.

\section{Statement of~results}\label{es02}

In~what follows, the~expression $\mathbb{E} = \langle B,t;\,t^{-1}Ht = K,\,\varphi \rangle$ means that $\mathbb{E}$ is~the~HNN-ex\-ten\-sion of~a~group~$B$ with~a~stable letter~$t$ and~subgroups~$H$ and~$K$ associated according to~an~isomorphism $\varphi\colon H \to K$. Let~us say that the~group~$\mathbb{E}$ \emph{satisfies~$(*)$}~if

1)\hspace{1ex}the~subgroups~$H$ and~$K$ are~normal in~$B$;

2)\hspace{1ex}the~subgroup $L = H \cap K$ is~$\varphi$\nobreakdash-in\-var\-i\-ant, i.e.,~$\varphi|_{L} \in \operatorname{Aut}L$.

If $X$ is~a~group and~$Y$ is~a~normal subgroup of~$X$, then the~restriction to~$Y$ of~any inner automorphism of~$X$ is~an~automorphism of~$Y$. The~set of~all such automorphisms is~a~subgroup of~$\operatorname{Aut}Y$, which we denote below by~$\operatorname{Aut}_{X}(Y)$. Let~us note that, if~$\mathbb{E}$ satisfies~$(*)$, then the~following is~defined:
\begin{list}{}{\topsep=0pt\itemsep=0pt\labelsep=1ex\labelwidth=2ex\leftmargin=\parindent\addtolength{\leftmargin}{\labelsep}\addtolength{\leftmargin}{\labelwidth}}
\item[{\makebox[1.2ex][l]{a}})]the~subgroups $\mathfrak{H} = \operatorname{Aut}_{B}(H)$, $\mathfrak{K} = \varphi\operatorname{Aut}_{B}(K)\varphi^{-1}$, and~$\mathfrak{U} = \operatorname{sgp}\{\mathfrak{H},\,\mathfrak{K}\}$ of~$\operatorname{Aut}H$;
\item[{\makebox[1.2ex][l]{b}})]the~subgroups $\mathfrak{L} = \operatorname{Aut}_{B}(L)$, $\mathfrak{F} = \operatorname{sgp}\{\varphi|_{L}\}$, and~$\mathfrak{V} = \operatorname{sgp}\{\mathfrak{L},\,\mathfrak{F}\}$ of~$\operatorname{Aut}L$.
\end{list}
Throughout the~paper, it~is~assumed that, if $\mathbb{E} = \langle B,t;\,t^{-1}Ht = K,\,\varphi \rangle$, then the~symbols~$L$,~$\mathfrak{H}$, $\mathfrak{K}$, $\mathfrak{U}$, $\mathfrak{L}$, $\mathfrak{F}$, and~$\mathfrak{V}$ are~defined as~above.

\begin{etheorem}\label{et01}
Suppose that the~group $\mathbb{E} = \langle B,t;\,t^{-1}Ht = K,\,\varphi \rangle$ satisfies~$(*)$ and~$\mathcal{C}$ is~a~root class of~groups closed under taking quotient groups. If~$B \in \mathcal{C}$\textup{,} then the~following statements are~equivalent and~any of~them implies that $\mathbb{E}$ is~residually a~$\mathcal{C}$\nobreakdash-group.

\textup{1.\hspace{1ex}}There exists a~homomorphism of~$\mathbb{E}$ onto~a~group from~$\mathcal{C}$ acting injectively on~the~subgroup~$B$.

\textup{2.\hspace{1ex}}The~inclusions $\mathfrak{U},\,\mathfrak{V} \in \mathcal{C}$ hold.
\end{etheorem}

Theorem~\ref{et01} admits the~following generalization.

\begin{etheorem}\label{et02}
Suppose that the~group $\mathbb{E} = \langle B,t;\,t^{-1}Ht = K,\,\varphi \rangle$ satisfies~$(*)$ and~$\mathcal{C}$ is~a~root class of~groups closed under taking quotient groups. If~$B$ has a~homomorphism~$\sigma$ onto a~group from~$\mathcal{C}$ acting injectively on~the~subgroup~$HK$\textup{,} then the~following statements hold.

\textup{1.\hspace{1ex}}The~condition $\mathfrak{U},\,\mathfrak{V} \in \mathcal{C}$ is~equivalent to~the~existence of~a~homomorphism of~$\mathbb{E}$ onto~a~group from~$\mathcal{C}$ that extends~$\sigma$.

\textup{2.\hspace{1ex}}If $\mathfrak{U},\,\mathfrak{V} \in \mathcal{C}$ and~$B$ is~residually a~$\mathcal{C}$\nobreakdash-group\textup{,} then $\mathbb{E}$ is~also residually a~$\mathcal{C}$\nobreakdash-group.
\end{etheorem}

Let~us note that the~assertions like Theorem~\ref{et02} can be~useful for~proving new first-lev\-el results (see, e.g.,~\cite{SokolovTumanova2023SMJ}).

\begin{ecorollary}\label{ec01}
Suppose that the~group $\mathbb{E} = \langle B,t;\,t^{-1}Ht = K,\,\varphi \rangle$ satisfies~$(*)$\textup{,} $\mathcal{C}$ is~a~root class of~groups closed under taking quotient groups\textup{,} and~the~subgroups~$H$ and~$K$ are~finite. Then $\mathbb{E}$ is~residually a~$\mathcal{C}$\nobreakdash-group if and~only if $B$ is~residually a~$\mathcal{C}$\nobreakdash-group and~$\mathfrak{U},\,\mathfrak{V} \in \mathcal{C}$.
\end{ecorollary}

Theorem~\ref{et01} requires that $B$ belongs to~$\mathcal{C}$. Theorem~\ref{et02} removes this restriction, but~still implicitly assumes that $H,\,K \in \mathcal{C}$. Now let~us consider the~case where $B$, $H$, and~$K$ are~arbitrary residually $\mathcal{C}$\nobreakdash-groups. We~start with~some necessary conditions for~$\mathbb{E}$ to~be~residually a~$\mathcal{C}$\nobreakdash-group.

\begin{etheorem}\label{et03}
Suppose that the~group $\mathbb{E}\kern-.5pt{} =\kern-.5pt{} \langle B,t;\,t^{-1}Ht\kern-.5pt{} =\kern-.5pt{} K,\,\varphi \rangle$ satisfies~$(*)$ and~$\mathcal{C}$~is~a~class of~groups closed under taking subgroups\textup{,} quotient groups\textup{,} and~direct products of~a~finite number of~factors. Suppose also that at~least one of~the~following statements holds\textup{:}
\begin{list}{}{\topsep=0pt\itemsep=0pt\labelsep=1ex\labelwidth=3ex\leftmargin=\parindent\addtolength{\leftmargin}{\labelsep}\addtolength{\leftmargin}{\labelwidth}}
\item[\textup{(\makebox[1.25ex]{$a$})}]$\mathfrak{U} = \mathfrak{H}$ or~$\mathfrak{U} = \mathfrak{K}$\textup{;}

\item[\textup{(\makebox[1.25ex]{$b$})}]$H$ and~$K$ satisfy a~non-trivi\-al identity\textup{;}

\item[\textup{(\makebox[1.25ex]{$c$})}]$\mathfrak{H}$ satisfies a~non-trivi\-al identity\textup{;}

\item[\textup{(\makebox[1.25ex]{$d$})}]$\mathfrak{K}$ satisfies a~non-trivi\-al identity.
\end{list}
If $\mathbb{E}$ is~residually a~$\mathcal{C}$\nobreakdash-group\textup{,} then the~quotient groups~$B/H$ and~$B/K$ have the~same property.
\end{etheorem}

Let~us note that the~papers~\cite{KuvaevSokolov2017IVM, Sokolov2021SMJ1} contain several more necessary conditions for~the residual $\mathcal{C}$\nobreakdash-ness of~HNN-ex\-ten\-sions, which are~similar to~Theorem~\ref{et03}.

Given a~class of~groups~$\mathcal{C}$ and~a~group~$X$, we~denote by~$\mathcal{C}^{*}(X)$ the~family of~normal subgroups of~$X$ defined as~follows: $N \in \mathcal{C}^{*}(X)$ if and~only if $X/N \in \mathcal{C}$. Let~us say that $X$ is~\emph{$\mathcal{C}$\nobreakdash-qua\-si-reg\-u\-lar} with~respect to~its subgroup~$Y$ if, for~each subgroup $M \in \mathcal{C}^{*}(Y)$, there exists a~subgroup $N \in \mathcal{C}^{*}(X)$ such that $N \cap Y \leqslant M$. The~property of~$\mathcal{C}$\nobreakdash-quasi-regularity is~closely related to~the~classical notion of~a~$\mathcal{C}$\nobreakdash-sep\-a\-ra\-ble subgroup~\cite{Malcev1958PIvPI} and~plays an~important role in~constructing the~kernels of~the~homomorphisms that map a~free construction of~groups onto~the~groups from~$\mathcal{C}$. Therefore, it~is~often a~part of~conditions sufficient for~such a~construction to~be~residually a~$\mathcal{C}$\nobreakdash-group (see, e.g.,~\cite{KimTang1999CMB, SokolovTumanova2016SMJ, SokolovTumanova2023SMJ}). In~\cite{Sokolov2023JGT, BaranovSokolov2025SMJ, Sokolov2025IVM}, a~number of~situations are~described in~which a~group turns~out to~be~$\mathcal{C}$\nobreakdash-quasi-regular with~respect to~its subgroup.

\begin{etheorem}\label{et04}
Suppose that the~group $\mathbb{E} = \langle B,t;\,t^{-1}Ht = K,\,\varphi \rangle$ satisfies~$(*)$\textup{,} $\mathcal{C}$ is~a~root class of~groups closed under taking quotient groups\textup{,} and~at~least one of~the~following statements holds\textup{:}
\begin{list}{}{\topsep=0pt\itemsep=0pt\labelsep=1ex\labelwidth=3.2ex\leftmargin=\parindent\addtolength{\leftmargin}{\labelsep}\addtolength{\leftmargin}{\labelwidth}}
\item[$(\alpha)$]$H/L \in \mathcal{C}$\textup{;}
\item[$(\beta)$]there exists a~homomorphism of~$B$ onto~a~group from~$\mathcal{C}$ acting injectively on~$L$.
\end{list}
Suppose also that $\mathfrak{U},\,\mathfrak{V} \in \mathcal{C}$ and~$B$ is~$\mathcal{C}$\nobreakdash-qua\-si-reg\-u\-lar with~respect to~$HK$. If~$B/H$ and~$B/K$ are~residually $\mathcal{C}$\nobreakdash-groups\textup{,} then $\mathbb{E}$ and~$B$ are~residually $\mathcal{C}$\nobreakdash-groups simultaneously.{\parfillskip=0pt{}\par}
\end{etheorem}

Let~us note that, if~$\mathcal{C}$ is~a~root class of~groups closed under taking quotient groups, while the~HNN-ex\-ten\-sion~$\mathbb{E}$ is~residually a~$\mathcal{C}$\nobreakdash-group and~satisfies~$(*)$, then, by~Proposition~\ref{ep312} below, $U$ and~$V$ are~residually $\mathcal{C}$\nobreakdash-groups and,~therefore, $\mathfrak{H}$, $\mathfrak{K}$, $\mathfrak{L}$, $\mathfrak{F}$, $\mathfrak{U}$, and~$\mathfrak{V}$ belong to~$\mathcal{C}$ when they are~finite. However, in~general, neither the~inclusions $\mathfrak{U} \in \mathcal{C}$ and~$\mathfrak{V} \in \mathcal{C}$, which appear in~Theorems~\ref{et01},~\ref{et02}, and~\ref{et04}, nor the~weaker conditions $\mathfrak{H} \in \mathcal{C}$, $\mathfrak{K} \in \mathcal{C}$, $\mathfrak{L} \in \mathcal{C}$, and~$\mathfrak{F} \in \mathcal{C}$ are~necessary for~$\mathbb{E}$ to~be~residually a~$\mathcal{C}$\nobreakdash-group, as~the~following example shows.

\begin{eexample}
Suppose that
\[
E = \langle a,b,c,t;\,[a,b] = [b,c] = [b,t] = 1,\,c^{-1}ac = t^{-1}at = ab \rangle,
\]
$A = \operatorname{sgp}\{a,b\}$, and~$\varphi$ is~the~automorphism of~$A$ taking~$a$ to~$ab$ and~$b$ to~$b$. Then $E$ is the~HNN-ex\-ten\-sion of~the~group
\[
B = \langle a,b,c;\,[a,b] = [b,c] = 1,\,c^{-1}ac = ab \rangle
\]
with the~coinciding subgroups $H = A = K$ associated according to~$\varphi$. The~group~$B$, in~turn, is~the~extension of~the~free abelian group~$A$ by~the~infinite cyclic group with the~generator~$c$, the~conjugation by~which acts on~$A$ as~$\varphi$. Therefore,
\[
\operatorname{Aut}_{B}(H) = \operatorname{Aut}_{B}(K) = \operatorname{Aut}_{B}(H \cap K) = \operatorname{sgp}\{\varphi|_{H \cap K}\}
\]
is~the~infinite cyclic group generated by~$\varphi$. At~the~same time, $E$ splits as~the~generalized free product of~the~isomorphic polycyclic groups~$B$~and
\[
D = \langle a,b,t;\,[a,b] = [b,t] = 1,\,t^{-1}at = ab \rangle
\]
with the~normal amalgamated subgroup~$A$. Hence, it~is~residually finite by~Theorem~9 from~\cite{BaumslagG1963TAMS1}.
\end{eexample}

The~next two theorems (and~Propositions~\ref{ep605}\nobreakdash--\ref{ep608} in~Section~\ref{es06}) describe some cases where the~condition $\mathfrak{U},\,\mathfrak{V} \in \mathcal{C}$ from~Theorem~\ref{et04} can be~modified or~weakened.

\begin{etheorem}\label{et05}
Suppose that the~group $\mathbb{E} = \langle B,t;\,t^{-1}Ht = K,\,\varphi \rangle$ satisfies~$(*)$\textup{,} $\mathcal{C}$ is~a~root class of~groups closed under taking quotient groups\textup{,} and~at~least one of~Statements~$(\alpha)$ and~$(\beta)$ from~Theorem~\textup{\ref{et04}} holds. Suppose also that $\mathfrak{U} = \mathfrak{H}$ or~$\mathfrak{U} = \mathfrak{K}$\textup{,} $\mathfrak{V} \in \mathcal{C}$\textup{,} and~$B$ is~$\mathcal{C}$\nobreakdash-qua\-si-reg\-u\-lar with~respect to~$HK$. Then $\mathbb{E}$ is~residually a~$\mathcal{C}$\nobreakdash-group if and~only if the~groups~$B$\textup{,} $B/H$\textup{,} and~$B/K$ have the~same property.
\end{etheorem}

\begin{etheorem}\label{et06}
Suppose that the~group $\mathbb{E} = \langle B,t;\,t^{-1}Ht = K,\,\varphi \rangle$ satisfies~$(*)$ and~$\mathcal{C}$ is~a~root class of~groups consisting only of~periodic groups and~closed under taking quotient groups. If~$H$ and~$K$ are~locally cyclic groups\textup{,} then the~following statements hold.

\textup{1.\hspace{1ex}}Let $B/H$ and~$B/K$ be~residually $\mathcal{C}$\nobreakdash-groups. Then the~group~$H/L$ is~finite if and~only if $(\alpha)$ holds.

\textup{2.\hspace{1ex}}Let $B$ be~residually a~$\mathcal{C}$\nobreakdash-group. Then the~group~$L$ is~finite if and~only if $(\beta)$ holds.

\textup{3.\hspace{1ex}}Suppose that $H/L$ is~finite\textup{,} $\mathfrak{F} \in \mathcal{C}$\textup{,} and~$B$ is~$\mathcal{C}$\nobreakdash-qua\-si-reg\-u\-lar with~respect to~$L$. Then $\mathbb{E}$ is~residually a~$\mathcal{C}$\nobreakdash-group if and~only if the~groups~$B$\textup{,} $B/H$\textup{,} and~$B/K$ have the~same property.

\textup{4.\hspace{1ex}}Suppose that $L$ is~finite and~$B$ is~$\mathcal{C}$\nobreakdash-qua\-si-reg\-u\-lar with~respect to~$HK$. Then $\mathbb{E}$ is residually a~$\mathcal{C}$\nobreakdash-group if and~only if $B$\textup{,} $B/H$\textup{,} and~$B/K$ are~residually $\mathcal{C}$\nobreakdash-groups and~$\mathfrak{F} \in \mathcal{C}$.
\end{etheorem}

Let~us note that the~condition ``$\mathfrak{U} = \mathfrak{H}$ or~$\mathfrak{U} = \mathfrak{K}$'' holds if at~least one of~the~subgroups~$H$ and~$K$ lies in~the~center of~$B$. In~this case, the~subgroup~$L$ is~certainly central in~$B$, whence $\mathfrak{L} = 1$ and~$\mathfrak{V} = \mathfrak{F}$. Therefore, Theorem~\ref{et05} generalizes Theorem~5 from~\cite{SokolovTumanova2023SMJ}. As~a~comment to~Theorem~\ref{et06}, we~also note that, if~$p$ is~a~prime and~$\mathcal{F}_{p}$ is~the~class of~finite $p$\nobreakdash-groups, then every residually $\mathcal{F}_{p}$\nobreakdash-group is~$\mathcal{F}_{p}$\nobreakdash-qua\-si-reg\-u\-lar with~respect to~any of~its locally cyclic subgroups~\cite[Theorem~3]{Sokolov2025IVM}.

Given a~class of~groups $\mathcal{C}$ consisting only of~periodic groups, let~us denote by~$\mathfrak{P}(\mathcal{C})$ the~set of~primes defined as~follows: $p \in \mathfrak{P}(\mathcal{C})$ if and~only if there exists a~$\mathcal{C}$\nobreakdash-group~$Z$ such that $p$ divides the~order of~some element of~$Z$. A~subgroup~$Y$ of~a~group~$X$ is~said to~be~\emph{$\mathfrak{P}(\mathcal{C})^{\prime}$\nobreakdash-iso\-lat\-ed} in~this group if, for~any element $x \in X$ and~for~any prime $q \notin \mathfrak{P}(\mathcal{C})$, it~follows from~the~inclusion $x^{q} \in Y$ that $x \in Y$. Clearly, if~$\mathfrak{P}(\mathcal{C})$ contains all prime numbers, then every subgroup is~$\mathfrak{P}(\mathcal{C})^{\prime}$\nobreakdash-iso\-lat\-ed.

Following~\cite{Sokolov2023JGT}, we~say that

--\hspace{1ex}an~abelian group is~\emph{$\mathcal{C}$\nobreakdash-bound\-ed} if, for~any quotient group~$B$ of~$A$ and~for~any $p \in \mathfrak{P}(\mathcal{C})$, the~$p$\nobreakdash-power torsion subgroup of~$B$ has a~finite exponent and~a~cardinality not~exceeding the~cardinality of~some $\mathcal{C}$\nobreakdash-group;

--\hspace{1ex}a~nilpotent group is~\emph{$\mathcal{C}$\nobreakdash-bound\-ed} if it has a~finite central series with~$\mathcal{C}$\nobreakdash-bound\-ed abelian factors.

It~is~easy to~see that, if~$\mathcal{C}$ is~a~root class of~groups consisting only of~periodic groups, then every finitely generated abelian group is~$\mathcal{C}$\nobreakdash-bound\-ed abelian and,~therefore, all finitely generated nilpotent groups are~$\mathcal{C}$\nobreakdash-bound\-ed nilpotent.

\begin{ecorollary}\label{ec02}
Suppose that the~group $\mathbb{E} = \langle B,t;\,t^{-1}Ht = K,\,\varphi \rangle$ satisfies~$(*)$ and~$\mathcal{C}$~is~a~root class of~groups consisting only of~periodic groups and~closed under taking quotient groups. Suppose also that $B$ is~a~$\mathcal{C}$\nobreakdash-bound\-ed nilpotent group and~at~least one of~Statements~$(\alpha)$ and~$(\beta)$ holds. Finally\textup{,} let at~least one of~the~following statements hold\textup{:}
\begin{list}{}{\topsep=0pt\itemsep=0pt\labelsep=1ex\labelwidth=3ex\leftmargin=\parindent\addtolength{\leftmargin}{\labelsep}\addtolength{\leftmargin}{\labelwidth}}
\item[\textup{(\makebox[1.25ex]{$a$})}]$\mathfrak{U},\,\mathfrak{V} \in \mathcal{C}$\textup{;}

\item[\textup{(\makebox[1.25ex]{$b$})}]$\mathfrak{U} = \mathfrak{H}$ or~$\mathfrak{U} = \mathfrak{K}$\textup{,} and~$\mathfrak{V} \in \mathcal{C}$\textup{;}

\item[\textup{(\makebox[1.25ex]{$c$})}]$H$ and~$K$ are~locally cyclic subgroups and~$\mathfrak{F} \in \mathcal{C}$.
\end{list}
Then $\mathbb{E}$ is~residually a~$\mathcal{C}$\nobreakdash-group if and~only if the~subgroups~$\{1\}$\textup{,} $H$\textup{,} and~$K$ are~$\mathfrak{P}(\mathcal{C})^{\prime}$\nobreakdash-iso\-lat\-ed in~$B$.
\end{ecorollary}

Let~us note that the~known results on~the~residual finiteness and~residual $p$\nobreakdash-fi\-nite\-ness of~HNN-ex\-ten\-sions do~not generalize the~assertions that follow from~Theorems~\ref{et04}\nobreakdash--\ref{et06} and~Corollary~\ref{ec02} when $\mathcal{C}$ is~the~class of~all finite groups or~finite $p$\nobreakdash-groups. Thus, these assertions are~also of~interest. Proofs of~the~formulated theorems and~corollaries are~given in~Sections~\ref{es04}\nobreakdash--\ref{es06}.

\section{Some auxiliary concepts and~facts}\label{es03}

We~use the~following notations throughout the~paper:
\begin{list}{}{\topsep=0pt\itemsep=0pt\labelsep=0ex\labelwidth=11ex\leftmargin=\parindent\addtolength{\leftmargin}{\parindent}\addtolength{\leftmargin}{\labelwidth}}
\item[{\makebox[11ex][l]{$\langle x \rangle$}}]the~cyclic group generated by~an~element~$x$;

\item[{\makebox[11ex][l]{$\widehat{x}$}}]the~inner automorphism produced by~an~element~$x$;

\item[{\makebox[11ex][l]{$[x,y]$}}]the~commutator of~elements~$x$ and~$y$, which is~equal to~$x^{-1}y^{-1}xy$;

\item[{\makebox[11ex][l]{$[X:Y]$}}]the~index of~a~subgroup~$Y$ in~a~group~$X$;

\item[{\makebox[11ex][l]{$\ker\sigma$}}]the~kernel of~a~homomorphism~$\sigma$;

\item[{\makebox[11ex][l]{$\operatorname{Im}\sigma$}}]the~image of~a~homomorphism~$\sigma$.
\end{list}

Let $\mathcal{C}$ be~a~class of~groups, and~let $X$ be~a~group. Following~\cite{Malcev1958PIvPI}, we~say that a~subgroup~$Y$ of~a~group~$X$ is~\emph{$\mathcal{C}$\nobreakdash-sep\-a\-ra\-ble} in~this group if, for~each element $x \in X \setminus Y$, there exists a~homomorphism~$\sigma$ of~$X$ onto~a~group from~$\mathcal{C}$ such that $x\sigma \notin Y\sigma$.

\begin{eproposition}\label{ep301}
\textup{\cite[Proposition~3]{SokolovTumanova2020IVM}}
Suppose that $\mathcal{C}$ is~a~class of~groups closed under taking quotient groups\textup{,} $X$~is~a~group\textup{,} and~$Y$ is~a~normal subgroup of~$X$. Then $Y$ is~$\mathcal{C}$\nobreakdash-sep\-a\-ra\-ble in~$X$ if and~only if $X/Y$ is~residually a~$\mathcal{C}$\nobreakdash-group.
\end{eproposition}

\begin{eproposition}\label{ep302}
\textup{\cite[Proposition~4]{SokolovTumanova2020IVM}}
Suppose that $\mathcal{C}$ is~a~class of~groups closed under taking subgroups\textup{,} $X$~is~a~group\textup{,} $Y$~is~a~subgroup of~$X$\textup{,} and~$Z \in \mathcal{C}^{*}(X)$. Then $Y \cap Z \in \mathcal{C}^{*}(Y)$ and\textup{,} if~$X$ is~residually a~$\mathcal{C}$\nobreakdash-group\textup{,} then $Y$ is~also residually a~$\mathcal{C}$\nobreakdash-group.
\end{eproposition}

\begin{eproposition}\label{ep303}
\textup{\cite[Proposition~2]{SokolovTumanova2020IVM}}
Suppose that $\mathcal{C}$ is~a~class of~groups closed under\linebreak taking subgroups and~direct products of~a~finite number of~factors. Then\textup{,} for~every~group~$X$\textup{,} the~following statements hold.

\textup{1.\hspace{1ex}}The~intersection of~finitely many subgroups of~the~family~$\mathcal{C}^{*}(X)$ is~again a~subgroup of~this family.

\textup{2.\hspace{1ex}}If $X$ is~residually a~$\mathcal{C}$\nobreakdash-group and~$Y$ is~a~finite subgroup of~$X$\textup{,} then there exists a~subgroup $N \in \mathcal{C}^{*}(X)$ that meets $Y$ trivially\textup{,} whence $Y \in \mathcal{C}$.
\end{eproposition}

\begin{eproposition}\label{ep304}
\textup{\cite[Proposition~4]{Tumanova2015IVM}}
Suppose that $\mathcal{C}$ is~a~class of~groups closed under taking quotient groups\textup{,} $X$~is~a~group\textup{,} and~$Y$ is~a~normal subgroup of~$X$. If~there exists a~homomorphism of~$X$ onto~a~group from~$\mathcal{C}$ acting injectively on~$Y$\textup{,} then $\operatorname{Aut}_{X}(Y) \in \mathcal{C}$.{\parfillskip=0pt{}\par}
\end{eproposition}

\begin{eproposition}\label{ep305}
If $\mathcal{C}$ is~a~class of~groups closed under taking quotient groups\textup{,} $X$~is~residually a~$\mathcal{C}$\nobreakdash-group\textup{,} and~$Y$ is~a~normal subgroup of~$X$\textup{,} then $\operatorname{Aut}_{X}(Y)$ is~also residually a~$\mathcal{C}$\nobreakdash-group.
\end{eproposition}

\begin{proof}
It~is~easy to~see that $\operatorname{Aut}_{X}(Y) \cong X/\mathcal{Z}_{X}(Y)$, where $\mathcal{Z}_{X}(Y)$ is~the~centralizer of~$Y$ in~$X$. Due~to~Proposition~\ref{ep301}, it~suffices to~show that, if~$\mathcal{Z}_{X}(Y) \ne X$, then the~subgroup~$\mathcal{Z}_{X}(Y)$ is~$\mathcal{C}$\nobreakdash-sep\-a\-ra\-ble in~$X$.

Let $x \in X \setminus \mathcal{Z}_{X}(Y)$. Then $[x,y] \ne 1$ for~some $y \in Y$. Since $X$ is~residually a~$\mathcal{C}$\nobreakdash-group, there exists a~subgroup $N \in \mathcal{C}^{*}(X)$ which does~not contain the~commutator~$[x,y]$. It~follows that $x \notin \mathcal{Z}_{X}(Y)N$ and,~hence, the~subgroup~$\mathcal{Z}_{X}(Y)$ is~$\mathcal{C}$\nobreakdash-sep\-a\-ra\-ble.
\end{proof}

The~proof of~the~following proposition is~quite simple and~is~therefore omitted.

\begin{eproposition}\label{ep306}
Suppose that $X$ is~a~group\textup{,} $Y$~is~a~normal subgroup of~$X$\textup{,} and~$\sigma$ is a~homomorphism of~$X$. Suppose also that $\bar\sigma\colon\!\operatorname{Aut}_{X}(Y) \to \operatorname{Aut}_{X\sigma}(Y\sigma)$ is~the~map taking~$\widehat{x}|_{Y}$ to~$\widehat{x\sigma}|_{Y\sigma}$ for~each $x \in X$. Then $\bar\sigma$ is~a~correctly defined surjective homomorphism.
\end{eproposition}

\begin{eproposition}\label{ep307}
Suppose that $\mathcal{C}$ is~a~class of~groups closed under taking quotient groups\textup{,} $X$~is~a~group\textup{,} and~$Y$ is~a~subgroup of~$X$. If~$X$ is~$\mathcal{C}$\nobreakdash-qua\-si-reg\-u\-lar with~respect to~$Y$ and~a~subgroup $M \in \mathcal{C}^{*}(Y)$ is~normal in~$X$\textup{,} then there exists a~subgroup $N \in \mathcal{C}^{*}(X)$ such that $N \cap Y = M$.
\end{eproposition}

\begin{proof}
Suppose that a~subgroup $M \in \mathcal{C}^{*}(Y)$ is~normal in~$X$. Since the~latter is~$\mathcal{C}$\nobreakdash-qua\-si-reg\-u\-lar with~respect to~$Y$, there exists a~subgroup $T \in \mathcal{C}^{*}(X)$ such that $T \cap Y \leqslant M$. Let $N = MT$. Then $N$ is~normal in~$X$ and~$N \cap Y = M$, as~is easy to~see. Since $\mathcal{C}$ is closed under taking quotient groups, it~follows from~the~relations $X/N \cong (X/T)/(N/T)$ and~$X/T \in \mathcal{C}$ that $X/N \in \mathcal{C}$. Thus, $N$ is~the~desired subgroup.
\end{proof}

\begin{eproposition}\label{ep308}
If $\mathcal{C}$ is~a~root class of~groups consisting only of~periodic groups\textup{,} then the~following statements hold.

\textup{1.\hspace{1ex}}Every $\mathcal{C}$-group is~of~finite exponent~\textup{\cite[Proposition~17]{SokolovTumanova2020IVM}}.

\textup{2.\hspace{1ex}}A~finite solvable group belongs to~$\mathcal{C}$ if and~only if its order is~a~$\mathfrak{P}(\mathcal{C})$\nobreakdash-num\-ber \textup{(}i.e.\textup{,}~each prime divisor of~this order lies in~$\mathfrak{P}(\mathcal{C})$\textup{)}~\textup{\cite[Proposition~8]{Tumanova2019SMJ}}.
\end{eproposition}

In~what follows, the~expression
\[
\mathbb{P} = \langle A * B;\,H = K,\,\varphi \rangle
\]
means that $\mathbb{P}$ is~the~generalized free product of~groups $A$ and~$B$ with~subgroups $H \leqslant A$ and~$K \leqslant B$ amalgamated by~an~isomorphism $\varphi\colon H \to K$. According to~\cite{BaumslagG1963TAMS1}, subgroups $R \leqslant A$ and~$S \leqslant B$ are~said to~be~\emph{$(H,K,\varphi)$-com\-pat\-i\-ble} if $(R \cap H)\varphi = S \cap K$. Suppose that $R$ is~normal in~$A$, $S$~is~normal in~$B$, and~$\varphi_{R,S}\colon HR/R \to KS/S$ is~the~map taking an~element~$hR$, $h \in H$, to~the~element~$(h\varphi)S$. It~follows from~the~equality $(R \cap H)\varphi = S \cap K$ that $\varphi_{R,S}$ is~a~correctly defined isomorphism and,~therefore, we~can consider the~generalized free product
\[
\mathbb{P}_{R,S} = \langle A/R * B/S;\,HR/R = KS/S,\,\varphi_{R,S} \rangle.
\]
As~is~easy to~see, the~identity mapping of~the~generators of~$\mathbb{P}$ into~$\mathbb{P}_{R,S}$ defines a~surjective homomorphism $\rho_{R,S}\colon \mathbb{P} \to \mathbb{P}_{R,S}$, whose kernel coincides with~the~normal closure of~the~set $R \cup S$ in~$\mathbb{P}$. We~note also that, if~$H$ and~$K$ are~normal in~$A$ and~$B$, respectively, then $H$ is~normal in~$\mathbb{P}$ and,~therefore, the~group~$\operatorname{Aut}_{\mathbb{P}}(H)$ is~defined. Clearly, this group is~generated by~its subgroups~$\operatorname{Aut}_{A}(H)$ and~$\varphi\operatorname{Aut}_{B}(K)\varphi^{-1}$.

Suppose that $x \in \mathbb{P}$ and~$x = x_{1}x_{2} \ldots x_{n}$, where $n \geqslant 1$ and~$x_{1},\,x_{2},\,\ldots,\,x_{n} \in A \cup B$. This product is~called a~\emph{reduced form} of~$x$ if no two adjacent factors~$x_{i}$ and~$x_{i+1}$ lie simultaneously in~$A$ or~$B$. The~number~$n$ is~said to~be~the~\emph{length} of~this form. It~is~known that, if~an~element $x \in \mathbb{P}$ has at~least one reduced form of~length greater than~$1$, then it is~non-trivi\-al (see, e.g.,~\cite[Chapter~IV, Theorem~2.6]{LyndonSchupp1977}).

Similar assertions hold for~the~HNN-ex\-ten\-sion $\mathbb{E} = \langle B,t;\,t^{-1}Ht = K,\,\varphi \rangle$. A~subgroup~$Q$ of~$B$ is~said to~be~\emph{$(H,K,\varphi)$-com\-pat\-i\-ble} if $(Q \cap H)\varphi = Q \cap K$. When $Q$ is~normal in~$B$, this equality ensures that the~map $\varphi_{Q}\colon HQ/Q \to KQ/Q$ given by~the~rule $hQ \mapsto (h\varphi)Q$, $h \in H$, is~a~correctly defined isomorphism. Therefore, the~HNN-ex\-ten\-sion
\[
\mathbb{E}_{Q} = \langle B/Q,t;\,t^{-1}(HQ/Q)t = KQ/Q,\,\varphi_{Q} \rangle
\]
can be~considered. As~above, the~identity mapping of~the~generators of~$\mathbb{E}$ into~$\mathbb{E}_{Q}$ defines a~surjective homomorphism $\rho_{Q}\colon \mathbb{E} \to \mathbb{E}_{Q}$, whose kernel coincides with~the~normal closure of~$Q$ in~$\mathbb{E}$.

Obviously, each element $x\kern-1pt{} \in \mathbb{E}$ can be~represented as~a~product $x\kern-.5pt{} =\kern-.5pt{} x_{0}t^{\varepsilon_{1}}x_{1} \ldots x_{n-1}t^{\varepsilon_{n}}x_{n}$, where $n \geqslant 0$, $x_{0},\,x_{1},\,\ldots,\,x_{n} \in B$, and~$\varepsilon_{1},\,\ldots,\,\varepsilon_{n} \in \{1,-1\}$. This product is~said to~be~a~\emph{reduced form} of~$x$ of~\emph{length~$n$} if, for~each $i \in \{1,\,\ldots,\,n-1\}$, the~equalities $-\varepsilon_{i} = 1 = \varepsilon_{i+1}$ imply that $x_{i} \notin H$, while the~equalities $\varepsilon_{i} = 1 = -\varepsilon_{i+1}$ guarantee that $x_{i} \notin K$. Britton's lemma~\cite{Britton1963AM} states that, if~an~element $x \in \mathbb{E}$ has a~reduced form of~non-zero length, then it is~non-trivi\-al. The~next two propositions are~special cases of~Theorem~4 from~\cite{Cohen1974JAMS} and~Theorem~1 from~\cite{Sokolov2023LJM}.

\begin{eproposition}\label{ep309}
Let $\mathbb{E} = \langle B,t;\,t^{-1}Ht = K,\,\varphi \rangle$. If~$N$ is~a~normal subgroup of~$\mathbb{E}$ and $N \cap B = 1$\textup{,} then $N$ is~free.
\end{eproposition}

\begin{eproposition}\label{ep310}
Let $\mathcal{C}$ be~a~root class of~groups. If~$\mathbb{E} = \langle B,t;\,t^{-1}Ht = K,\,\varphi \rangle$\textup{,} $B$~is residually a~$\mathcal{C}$\nobreakdash-group\textup{,} and~there exists a~homomorphism of~$\mathbb{E}$ onto~a~group from~$\mathcal{C}$ acting injectively on~$H$ and~$K$\textup{,} then $\mathbb{E}$ is~residually a~$\mathcal{C}$\nobreakdash-group.
\end{eproposition}

In~what follows, if~$\mathbb{E} = \langle B,t;\,t^{-1}Ht = K,\,\varphi \rangle$, then the~expression
\[
\mathbb{B} = \langle B * B;\,H = K,\,\varphi \rangle
\]
means that $\mathbb{B}$ is~the~generalized free product of~two isomorphic copies of~$B$ with~the~subgroups~$H$ and~$K$ amalgamated by~the~same isomorphism $\varphi\colon H \to K$. Let $\zeta_{gen}$ be~the~map of~the~generators of~$\mathbb{B}$ into~$\mathbb{E}$ given by~the~rule: $x \mapsto t^{-1}xt$, $y \mapsto y$, where $x$ and~$y$ are~generators of~the~first and~second instances of~$B$, respectively. Clearly, when extended to~a~mapping of~words, $\zeta_{gen}$~takes all defining relations of~$\mathbb{B}$ to~the~equalities valid in~$\mathbb{E}$ and,~therefore, induces a~homomorphism $\zeta\colon \mathbb{B} \to \mathbb{E}$. It~is~also easy to~see that, if~$x_{1} \ldots x_{n}$ is~a~reduced form of~an~element $x \in \mathbb{B} \setminus \{1\}$, then the~product $x_{1}\zeta \ldots x_{n}\zeta$ is~a~reduced form of~the~element~$x\zeta$ and~$x\zeta \ne 1$. Hence, $\zeta$~is~injective. It~can also be~noted that, if~$\mathbb{E}$ satisfies~$(*)$, then $\mathfrak{U} = \operatorname{Aut}_{\mathbb{B}}(H)$.

\begin{eproposition}\label{ep311}
Suppose that the~group $\mathbb{E} = \langle B,t;\,t^{-1}Ht = K,\,\varphi \rangle$ satisfies~$(*)$ and~$\mathcal{C}$~is a~class of~groups closed under taking subgroups and~quotient groups. If~there exists a~homomorphism~$\sigma$ of~$\mathbb{E}$ onto~a~group from~$\mathcal{C}$ acting injectively on~$H$ and~$K$\textup{,} then $\mathfrak{U},\,\mathfrak{V} \in\nolinebreak \mathcal{C}$.
\end{eproposition}

\begin{proof}
It~is~obvious that $\mathfrak{V} = \operatorname{Aut}_{\mathbb{E}}(L)$. Therefore, the~inclusion $\mathfrak{V} \in \mathcal{C}$ follows from~Proposition~\ref{ep304}. Let $\mathbb{B} = \langle B * B;\,H = K,\,\varphi \rangle$, and~let $\zeta\colon \mathbb{B} \to \mathbb{E}$ be~the~homomorphism defined above. Since $\sigma$ is~injective on~$K = K\zeta$ and~$\mathcal{C}$ is~closed under taking subgroups, $\mathbb{B}$~has a~homomorphism onto~a~group from~$\mathcal{C}$ acting injectively on~$H$ and~$K$. Hence, $\mathfrak{U} = \operatorname{Aut}_{\mathbb{B}}(H) \in \mathcal{C}$ by~the~same Proposition~\ref{ep304}.
\end{proof}

\begin{eproposition}\label{ep312}
Suppose that the~group $\mathbb{E} = \langle B,t;\,t^{-1}Ht = K,\,\varphi \rangle$ satisfies~$(*)$ and~$\mathcal{C}$~is a~root class of~groups closed under taking quotient groups. If~$\mathbb{E}$ is~residually a~$\mathcal{C}$\nobreakdash-group\textup{,} then $\mathfrak{U}$ and~$\mathfrak{V}$ are~also residually $\mathcal{C}$\nobreakdash-groups and\textup{,}~therefore\textup{,} the~groups~$\mathfrak{H}$\textup{,}~$\mathfrak{K}$\textup{,} $\mathfrak{L}$\textup{,} $\mathfrak{F}$\textup{,} $\mathfrak{U}$\textup{,} and~$\mathfrak{V}$ belong to~$\mathcal{C}$ when they are~finite.
\end{eproposition}

\begin{proof}
As~noted above, the~group $\mathbb{B} = \langle B * B;\,H = K,\,\varphi \rangle$ can be~embedded into~the~residually $\mathcal{C}$\nobreakdash-group~$\mathbb{E}$. Therefore, it~is~itself residually a~$\mathcal{C}$\nobreakdash-group by~Proposition~\ref{ep302}. Since $\mathfrak{U} = \operatorname{Aut}_{\mathbb{B}}(H)$ and~$\mathfrak{V} = \operatorname{Aut}_{\mathbb{E}}(L)$, Proposition~\ref{ep305} implies that $\mathfrak{U}$ and~$\mathfrak{V}$ are~also residually $\mathcal{C}$\nobreakdash-groups. The~inclusions $\mathfrak{H},\,\mathfrak{K},\,\mathfrak{L},\,\mathfrak{F},\,\mathfrak{U},\,\mathfrak{V} \in \mathcal{C}$ follows from~Proposition~\ref{ep303}.
\end{proof}

\section{Proof of~Theorems~\ref{et01}--\ref{et02} and~Corollary~\ref{ec01}}\label{es04}

\begin{eproposition}\label{ep401}
\textup{\cite[Theorem~1]{Tumanova2015IVM}}
If $\mathcal{C}$ is~a~root class of~groups\textup{,} $\mathbb{P} = \langle A * B;\,H = K,\,\varphi \rangle$\textup{,} $H$~is~normal in~$A$\textup{,} $K$~is~normal in~$B$\textup{,} and~$A,\,B,\,A/H,\,B/K,\,\operatorname{Aut}_{\mathbb{P}}(H) \in \mathcal{C}$\textup{,} then there exists a~homomorphism of~$\mathbb{P}$ onto~a~group from~$\mathcal{C}$ acting injectively on~$A$ and~$B$.
\end{eproposition}

Let the~group $\mathbb{E} = \langle B,t;\,t^{-1}Ht = K,\,\varphi \rangle$ satisfy~$(*)$ and~$\mathbb{B} = \langle B * B;\,H = K,\,\varphi \rangle$. Then the~subgroup~$K$ of~the~first free factor of~$\mathbb{B}$ and~the~subgroup~$H$ of~the~second one are~$(H,K,\varphi)$-com\-pat\-i\-ble. Therefore, the~generalized free product
\[
\mathbb{B}_{K,H} = \big\langle (B/K)*(B/H);\,HK/K = KH/H,\,\varphi_{K,H} \big\rangle
\]
and~the~group~$\operatorname{Aut}_{\mathbb{B}_{K,H}}(HK/K)$ are~defined.

\begin{eproposition}\label{ep402}
Suppose that the~group $\mathbb{E} = \langle B,t;\,t^{-1}Ht = K,\,\varphi \rangle$ satisfies~$(*)$\textup{,} $\mathcal{C}$~is a~root class of~groups\textup{,} and~$H \cap K = 1$. If~$B/H,\,B/K,\,B/HK,\,\operatorname{Aut}_{\mathbb{B}_{K,H}}(HK/K) \in \mathcal{C}$\textup{,} then there exists a~homomorphism of~$\mathbb{E}$ onto~a~group from~$\mathcal{C}$ acting injectively on~$B$.
\end{eproposition}

\begin{proof}
If $\mathcal{C}$ contains non-pe\-ri\-od\-ic groups, we denote by~$\mathcal{I}$ the~additive group of~the~ring~$\mathbb{Z}$. Otherwise, let $\mathcal{I}$ be~the~additive group of~the~ring~$\mathbb{Z}_{n}$, where $n$ is~the~order of~some $\mathcal{C}$\nobreakdash-group and~$n \geqslant 4$. Since $\mathcal{C}$ is~closed under taking subgroups and~extensions, the~number~$n$ with~the~indicated properties exists and,~in~both cases, $\mathcal{I} \in \mathcal{C}$.

For each $i \in \mathcal{I}$, let $B_{i}$ denote an~isomorphic copy of~$B$. Let also $\beta_{i}\colon B \to B_{i}$ be~the~corresponding isomorphism, $H_{i} = H\beta_{i}$, and~$K_{i} = K\beta_{i}$. Consider the~group
\[
\mathfrak{P} = \langle B_{i};\,H_{i} = K_{i-1}\ (i \in \mathcal{I}) \rangle
\]
whose generators are~the~generators of~the~groups~$B_{i}$, $i \in \mathcal{I}$, and~whose defining relations are~those of~$B_{i}$, $i \in \mathcal{I}$, and~all possible relations of~the~form $h\varphi\beta_{i-1} = h\beta_{i}$, where $h \in H$ and~$i \in \mathcal{I}$. It~is~easy to~see that, if~$\mathcal{I}$ is~infinite, then $\mathfrak{P}$ is~the~tree product of~the~groups~$B_{i}$, $i \in \mathcal{I}$, that corresponds to~an~infinite chain. Otherwise, $\mathfrak{P}$~is~the~polygonal product of~the~same groups. Theorem~1 from~\cite{KarrasSolitar1970TAMS} says that, in~the~first case, the~identity mappings of~the~generators of~$B_{i}$, $i \in \mathcal{I}$, into~$\mathfrak{P}$ can be~extended to~injective homomorphisms. It~follows from~the~relations $H \cap K = 1$ and~$n \geqslant 4$ that the~same statement holds in~the~second case~\cite{AllenbyTang1989CMB}.

Let $\alpha_{gen}$ be~the~mapping of~the~generators of~$\mathfrak{P}$ acting as~the~isomorphisms~$\beta_{i}^{-1}\beta_{i+1}^{\vphantom{1}}$, $i \in \mathcal{I}$. Clearly, $\alpha_{gen}$~defines an~automorphism~$\alpha$ of~$\mathfrak{P}$, whose order is~equal to~the~order of~$\mathcal{I}$. Let $\mathfrak{Q}$ denote the~splitting extension of~$\mathfrak{P}$ by~the~cyclic group~$\langle \alpha \rangle$. Consider the~mapping~$\lambda_{gen}$ of~the~generators of~$\mathbb{E}$ into~$\mathfrak{Q}$ that acts on~the~generators of~$B$ as~$\beta_{0}$ and~takes~$t$ to~$\alpha$ (here and~below, we~identify the~groups~$B$, $B_{i}$, $i \in \mathcal{I}$, and~$\mathfrak{P}$ with~the~corresponding subgroups of~$\mathbb{E}$, $\mathfrak{P}$, and~$\mathfrak{Q}$, respectively). It~is~easy to~see that, when extended to~a~mapping of~words, $\lambda_{gen}$~takes all defining relations of~$\mathbb{E}$ to~the~equalities valid in~$\mathfrak{Q}$. Therefore, it~induces a~homomorphism $\lambda\colon \mathbb{E} \to \mathfrak{Q}$, which acts on~$B$ as~$\beta_{0}$. Since $\mathfrak{Q}$ is~obviously generated by~the~set $\{\alpha\} \cup \{b\beta_{0} \mid b \in B\}$, the~homomorphism~$\lambda$ is~surjective.

Let $i \in \mathcal{I}$. It~is~clear that $\beta_{i+1}$ and~$\beta_{i}$ induce an~isomorphism~$\gamma_{i}$ of~$\mathbb{B}_{K,H}$ onto~the~generalized free product
\[
\mathfrak{P}_{i} = \big\langle (B_{i+1}/K_{i+1})*(B_{i}/H_{i});\,H_{i+1}K_{i+1}/K_{i+1} = K_{i}H_{i}/H_{i},\,\varphi_{i} \big\rangle,
\]
where the~isomorphism $\varphi_{i}\colon H_{i+1}K_{i+1}/K_{i+1} \to K_{i}H_{i}/H_{i}$ is~given by~the~rule:
\[
(h\beta_{i+1})K_{i+1} \mapsto (h\varphi \beta_{i})H_{i}, \quad h \in H.
\]
The~relations
\[
(B/H)/(KH/H) \cong B/HK \cong (B/K)/(HK/K)
\]
and~$B/HK \in \mathcal{C}$ mean that all conditions of~Proposition~\ref{ep401} hold for~the~group~$\mathbb{B}_{K,H}$. Since $(B/K)\gamma_{i} = B_{i+1}/K_{i+1}$ and~$(B/H)\gamma_{i} = B_{i}/H_{i}$, it~follows that there exists a~homomorphism~$\eta_{i}$ of~$\mathfrak{P}_{i}$ onto~a~group from~$\mathcal{C}$ satisfying the~equalities
\[
\ker\eta_{i} \cap B_{i}/H_{i} = 1 = \ker\eta_{i} \cap B_{i+1}/K_{i+1}.
\]
Consider the~mapping of~the~generators of~$\mathfrak{P}$ into~$\mathfrak{P}_{i}$ which acts identically on~the~elements of~$B_{i}$ and~$B_{i+1}$, and~takes the~generators of~other free factors to~$1$. It~is~easy to~see that this mapping defines a~surjective homomorphism $\theta_{i}\colon \mathfrak{P} \to \mathfrak{P}_{i}$. Therefore, if~$M_{i}$ denotes the~subgroup $\ker\theta_{i}\eta_{i}$, we~have $M_{i} \in \mathcal{C}^{*}(\mathfrak{P})$, $M_{i} \cap B_{i} = H_{i}$, and~$M_{i} \cap B_{i+1} = K_{i+1}$.

Let $M = M_{0} \cap M_{-1}$. It~follows from~Proposition~\ref{ep303} and~the~above that $M \in \mathcal{C}^{*}(\mathfrak{P})$~and
\[
M \cap B_{0} = (M_{0} \cap B_{0}) \cap (M_{-1} \cap B_{0}) = H_{0} \cap K_{0} = (H \cap K)\beta_{0} = 1.
\]
Since $\langle \alpha \rangle \cong \mathcal{I} \in \mathcal{C}$, the~factors of~the~subnormal sequence $M \leqslant \mathfrak{P} \leqslant \mathfrak{Q}$ belong to~$\mathcal{C}$. Hence, by~Gruenberg's condition, there exists a~subgroup $N \in \mathcal{C}^{*}(\mathfrak{Q})$ lying in~$M$. It~now follows from~the~relations $B_{0} = B\lambda$ and~$N \cap B_{0} \leqslant M \cap B_{0} = 1$ that the~composition of~$\lambda$ and~the~natural homomorphism $\mathfrak{Q} \to \mathfrak{Q}/N$ is~the~desired mapping.
\end{proof}

\begin{eproposition}\label{ep403}
Suppose that the~group $\mathbb{E} = \langle B,t;\,t^{-1}Ht = K,\,\varphi \rangle$ satisfies~$(*)$ and~$\mathcal{C}$~is a~root class of~groups. If~$L,\,B/H,\,B/K,\,B/HK,\,\operatorname{Aut}_{\mathbb{E}}(L),\,\operatorname{Aut}_{\mathbb{B}_{K,H}}(HK/K) \in \mathcal{C}$\textup{,} then there exists a~homomorphism of~$\mathbb{E}$ onto~a~group from~$\mathcal{C}$ acting injectively on~$B$.
\end{eproposition}

\begin{proof}
Since $L$ is~a~normal subgroup of~$B$ and~$(L \cap H)\varphi = L\varphi = L = L \cap K$, the~HNN-ex\-ten\-sion
\[
\mathbb{E}_{L} = \big\langle B/L,t;\,t^{-1}(H/L)t = K/L,\,\varphi_{L} \big\rangle
\]
is~defined. Consider the~following groups:
\begin{multline*}
\begin{aligned}
&\mathbb{B}_{K,H} = \big\langle (B/K)*(B/H);\,HK/K = KH/H,\,\varphi_{K,H} \big\rangle,\\
&\mathbb{B}_{K/L,H/L} = \big\langle (B/L)/(K/L) * (B/L)/(H/L);
\end{aligned}
\\
(H/L)(K/L)/(K/L) = (K/L)(H/L)/(H/L),\,\varphi_{K/L,H/L} \big\rangle.
\end{multline*}

It~is~clear that the~identity mapping of~the~generators of~$\mathbb{B}_{K/L,H/L}$ into~$\mathbb{B}_{K,H}$ defines an~isomorphism, which takes the~subgroup $(H/L)(K/L)/(K/L)$ onto~$HK/K$. Therefore,
\[
\operatorname{Aut}_{\mathbb{B}_{K/L,H/L}}\big((H/L)(K/L)/(K/L)\big) \cong \operatorname{Aut}_{\mathbb{B}_{K,H}}(HK/K) \in \mathcal{C}.
\]
Since
\begin{gather*}
(B/L)/(H/L) \cong B/H \in \mathcal{C}, \quad 
(B/L)/(K/L) \cong B/K \in \mathcal{C}, \\
(B/L)/(H/L)(K/L) \cong B/HK \in \mathcal{C}, \quad
H/L \cap K/L = 1,
\end{gather*}
the~HNN-ex\-ten\-sion $\mathbb{E}_{L}$ satisfies the~conditions of~Proposition~\ref{ep402}. Hence, there exists a~homomorphism $\tau_{L}$ of~$\mathbb{E}_{L}$ onto~a~group from~$\mathcal{C}$ which is~injective on~$B/L$. By~Proposition~\ref{ep309}, the~kernel of~$\tau_{L}$ is~a~free group.

Since $L$ is~normal in~$\mathbb{E}$, the~equality $L = \ker\rho_{L}$ holds. Therefore, the~subgroup $U = \ker\rho_{L}\tau_{L}$ is~an~extension of~$L$ by~a~free group. It~is~well known that such an~extension is~splittable, i.e.,~$U$ has a~free subgroup~$F$ satisfying the~relations $U = LF$ and~$L \cap F = 1$.

Let $\xi\colon \mathbb{E} \to \operatorname{Aut}L$ be~the~homomorphism taking an~element $x \in \mathbb{E}$ to~the~automorphism~$\widehat{x}|_{L}$. Its kernel obviously coincides with~the~centralizer~$\mathcal{Z}_{\mathbb{E}}(L)$ of~$L$ in~$\mathbb{E}$. Since $\mathcal{C}$ is~closed under taking subgroups and~extensions, it~follows from~this fact and~the~relations
\[
U = LF, \ \ L \cap F = 1, \ \ L \leqslant B, \ \ 
F\mathcal{Z}_{\mathbb{E}}(L)/\mathcal{Z}_{\mathbb{E}}(L) \leqslant \mathbb{E}/\mathcal{Z}_{\mathbb{E}}(L), \ \ 
\operatorname{Im}\xi = \operatorname{Aut}_{\mathbb{E}}(L) \in \mathcal{C}
\]
that the~subgroup $V = \mathcal{Z}_{\mathbb{E}}(L) \cap F$ is~normal in~$U$,
\begin{gather*}
F/V \cong F\mathcal{Z}_{\mathbb{E}}(L)/\mathcal{Z}_{\mathbb{E}}(L) \in \mathcal{C}, \quad LV/V \cong L/L \cap V \cong L \in \mathcal{C},\\
U/LV = LF/LV \cong F/V(L \cap F) = F/V \in \mathcal{C},
\end{gather*}
and~the~quotient group~$U/V$ belongs to~$\mathcal{C}$ as~an~extension of~$LV/V$ by~a~group isomorphic to~$U/LV$. In~addition, $U \in \mathcal{C}^{*}(\mathbb{E})$ due~to~the~definition of~$\tau_{L}$. Thus, we~can apply Gruenberg's condition to~the~subnormal series $1 \leqslant V \leqslant U \leqslant \mathbb{E}$ and~find a~subgroup $W \in \mathcal{C}^{*}(\mathbb{E})$ lying in~$V$.

Since $\tau_{L}$ acts injectively on~$B/L$, the~equality $U \cap B = L$ holds. It~now follows from~the~inclusions $W \leqslant V \leqslant F \leqslant U$ that $W \cap B \leqslant F \cap (U \cap B) = F \cap L = 1$. Hence, the~natural homomorphism $\mathbb{E} \to \mathbb{E}/W$ is~the~desired one.
\end{proof}

\begin{proof}[\textup{\textbf{Proof of~Theorem~\ref{et01}}}]
The~implication $1 \Rightarrow 2$ and~the~residual $\mathcal{C}$\nobreakdash-ness of~$\mathbb{E}$ follow from Propositions~\ref{ep311} and~\ref{ep310}, respectively. To~prove the~implication $2 \Rightarrow 1$, it~is~sufficient to~show that all conditions of~Proposition~\ref{ep403} hold if $\mathfrak{U},\,\mathfrak{V} \in \mathcal{C}$.

Indeed, since $\mathcal{C}$ is~closed under taking subgroups and~quotient groups, the~inclusions $L,\,B/H,\,B/K,\,B/HK \in \mathcal{C}$ follow from~the~condition $B \in \mathcal{C}$. Let $\mathbb{B} = \langle B * B;\,H = K,\,\varphi \rangle$. By~Proposition~\ref{ep306}, the~relations $\operatorname{Aut}_{\mathbb{B}}(H) = \mathfrak{U} \in \mathcal{C}$, $\mathbb{B}_{K,H} = \mathbb{B}\rho_{K,H}$, and~$HK/K = H\rho_{K,H}$ imply that $\operatorname{Aut}_{\mathbb{B}_{K,H}}(HK/K) \in \mathcal{C}$. It~remains to~note that $\operatorname{Aut}_{\mathbb{E}}(L) = \mathfrak{V} \in \mathcal{C}$.
\end{proof}

\begin{eproposition}\label{ep404}
Suppose that the~group $\mathbb{E} = \langle B,t;\,t^{-1}Ht = K,\,\varphi \rangle$ satisfies~$(*)$\textup{,} $Q$~is a~normal $(H,K,\varphi)$-com\-pat\-i\-ble subgroup of~$B$\textup{,}~and
\[
\mathbb{E}_{Q} = \langle B/Q,t;\,t^{-1}(HQ/Q)t = KQ/Q,\,\varphi_{Q} \rangle.
\]
Suppose also that the~symbols~$\mathfrak{H}_{Q}$\textup{,} $\mathfrak{K}_{Q}$\textup{,} and~$\mathfrak{L}_{Q}$ denote the~subgroups
\[
\operatorname{Aut}_{B/Q}^{\vphantom{1}}(HQ/Q), \quad 
\varphi_{Q}^{\vphantom{1}}\operatorname{Aut}_{B/Q}^{\vphantom{1}}(KQ/Q)\varphi_{Q}^{-1}, \quad 
\text{and} \quad \operatorname{Aut}_{B/Q}^{\vphantom{1}}(LQ/Q),
\]
respectively. Then the~following statements hold.

\textup{1.\hspace{1ex}}There exists a~homomorphism of~$\mathfrak{U}$ onto~the~group $\mathfrak{U}_{Q} = \operatorname{sgp}\{\mathfrak{H}_{Q},\,\mathfrak{K}_{Q}\}$ which maps the~subgroups~$\mathfrak{H}$ and~$\mathfrak{K}$ onto~$\mathfrak{H}_{Q}$ and~$\mathfrak{K}_{Q}$\textup{,} respectively.

\textup{2.\hspace{1ex}}The~subgroup~$LQ/Q$ is~$\varphi_{Q}$\nobreakdash-in\-var\-i\-ant and~there exists a~homomorphism of~$\mathfrak{V}$ onto the~group $\mathfrak{V}_{Q} = \operatorname{sgp}\{\mathfrak{L}_{Q},\,\varphi_{Q}|_{LQ/Q}\}$ which maps $\mathfrak{L}$ and~$\mathfrak{F}$ onto~the~subgroups~$\mathfrak{L}_{Q}$ and $\mathfrak{F}_{Q} = \langle \varphi_{Q}|_{LQ/Q} \rangle$\textup{,} respectively.
\end{eproposition}

\begin{proof}
1.\hspace{1ex}Let $\mathbb{B} = \langle B * B;\,H = K,\,\varphi \rangle$. Since $Q$ is~$(H,K,\varphi)$-com\-pat\-i\-ble, the~generalized free product
\[
\mathbb{B}_{Q,Q} = \langle B/Q * B/Q;\,HQ/Q = KQ/Q,\,\varphi_{Q,Q} \rangle
\]
is defined. It~follows from~Proposition~\ref{ep306} that the~map $\overline{\rho_{Q,Q}}\colon\! \operatorname{Aut}_{\mathbb{B}}(H) \to \operatorname{Aut}_{\mathbb{B}_{Q,Q}}(H\rho_{Q,Q})$ given by~the~rule $\widehat{x}|_{H} \mapsto \widehat{x\rho_{Q,Q}}|_{H\rho_{Q,Q}}$, $x \in \mathbb{B}$, is~a~surjective homomorphism. Clearly, $\mathfrak{U}_{Q} = \operatorname{Aut}_{\mathbb{B}_{Q,Q}}(HQ/Q)$, $\mathfrak{U} = \operatorname{Aut}_{\mathbb{B}}(H)$, $H\rho_{Q,Q} = HQ/Q$, $\mathfrak{H}\overline{\rho_{Q,Q}} = \mathfrak{H}_{Q}$, and~$\mathfrak{K}\overline{\rho_{Q,Q}} = \mathfrak{K}_{Q}$. Therefore, the~homomorphism $\overline{\rho_{Q,Q}}$ is~desired.

2.\hspace{1ex}The~equality $L\varphi = L$ and~the~definition of~$\varphi_{Q}$ imply that $(LQ/Q)\varphi_{Q} = LQ/Q$. Since $\varphi_{Q}$ is~an~isomorphism, this relation means that $\varphi_{Q}|_{LQ/Q} \in \operatorname{Aut}LQ/Q$. The~existence of~the~desired homomorphism is~ensured by~Proposition~\ref{ep306} due~to~the~equalities $\mathfrak{V}_{Q} =\nolinebreak \operatorname{Aut}_{\mathbb{E}_{Q}}(LQ/Q)$, $\mathfrak{V} = \operatorname{Aut}_{\mathbb{E}}(L)$, and~$LQ/Q = L\rho_{Q}$.
\end{proof}

\begin{proof}[\textup{\textbf{Proof of~Theorem~\ref{et02}}}]
Statement~2 follows from~Statement~1 and~Proposition~\ref{ep310}. Let us prove Statement~1. If~there exists a~homomorphism of~$\mathbb{E}$ onto~a~group from~$\mathcal{C}$ that extends~$\sigma$, then $\mathfrak{U},\,\mathfrak{V} \in \mathcal{C}$ by~Proposition~\ref{ep311}. It~remains to~show that the~converse also holds.

Let $Q = \ker\sigma$. Then $B/Q \in \mathcal{C}$ and~it follows from~the~equality $Q \cap HK = 1$ that $Q \cap H = 1 = Q \cap K$ and~$HQ/Q \cap KQ/Q = LQ/Q$. Hence, the~groups~$\mathbb{E}_{Q}$, $\mathfrak{U}_{Q}$ and~$\mathfrak{V}_{Q}$ can be~defined as~in~Proposition~\ref{ep404}. By~this proposition, the~subgroup~$LQ/Q$ is~$\varphi_{Q}$\nobreakdash-in\-var\-i\-ant. Since $\mathfrak{U},\,\mathfrak{V} \in \mathcal{C}$ and~$\mathcal{C}$ is~closed under taking quotient groups, Proposition~\ref{ep404} also implies that $\mathfrak{U}_{Q},\,\mathfrak{V}_{Q} \in \mathcal{C}$. By~Theorem~\ref{et01}, it~follows from~these relations and~the~equality $HQ/Q \cap\nolinebreak KQ/Q = LQ/Q$ that there exists a~homomorphism~$\tau$ of~the~group~$\mathbb{E}_{Q}$ which satisfies the~conditions $\ker\tau \cap B/Q = 1$ and~$\operatorname{Im}\tau \in \mathcal{C}$. Since $\rho_{Q}$ extends~$\sigma$, the~composition~$\rho_{Q}\tau$ is~the~desired mapping.
\end{proof}

\begin{proof}[\textup{\textbf{Proof of~Corollary~\ref{ec01}}}]
\textit{Necessity.}
Proposition~\ref{ep303} states that there exists a~homomorphism of~$\mathbb{E}$ onto~a~group from~$\mathcal{C}$ acting injectively on~the~finite subgroup~$HK$. Hence, $\mathfrak{U},\,\mathfrak{V} \in \mathcal{C}$ by~Proposition~\ref{ep311}. The~residual $\mathcal{C}$\nobreakdash-ness of~$B$ is~ensured by~Proposition~\ref{ep302}.

\textit{Sufficiency.}
As~above, if~$B$ is~residually a~$\mathcal{C}$\nobreakdash-group, it~has a~homomorphism onto~a~group from~$\mathcal{C}$ acting injectively on~$HK$. Therefore, we~can use Statement~2 of~Theorem~\ref{et02} to~prove the~residual $\mathcal{C}$\nobreakdash-ness of~$\mathbb{E}$.
\end{proof}

\section{Proof of~Theorem~\ref{et03}}\label{es05}

\begin{eproposition}\label{ep501}
\textup{\cite[Proposition~9]{SokolovTumanova2020IVM}}
Let $\mathbb{P} = \langle A * B;\,H = K,\,\varphi \rangle$. Suppose also that $H$ is~normal in~$A$\textup{,} $K$ is~normal in~$B$\textup{,} and~the~group~$\operatorname{Aut}_{\mathbb{P}}(H)$ coincides with~one of~its subgroups~$\operatorname{Aut}_{A}(H)$ and~$\varphi\operatorname{Aut}_{B}(H)\varphi^{-1}$. If~$\mathcal{C}$ is~a~class of~groups closed under taking subgroups and~$\mathbb{P}$ is~residually a~$\mathcal{C}$\nobreakdash-group\textup{,} then the~following statements hold.

\enlargethispage{10pt}

\textup{1.\hspace{1ex}}If $K \ne B$\textup{,} then $H$ is~$\mathcal{C}$\nobreakdash-sep\-a\-ra\-ble in~$A$.

\textup{2.\hspace{1ex}}If $H \ne A$\textup{,} then $K$ is~$\mathcal{C}$\nobreakdash-sep\-a\-ra\-ble in~$B$.
\end{eproposition}

\begin{eproposition}\label{ep502}
\textup{\cite[Theorem~1]{KuvaevSokolov2017IVM}}
Let $\mathbb{E} = \langle B,t;\,t^{-1}Ht = K,\,\varphi \rangle$. Suppose also that $\mathcal{C}$ is~a~class of~groups\textup{,} the~symbols~$\overline{H}$ and~$\overline{K}$ denote the~subgroups $\bigcap_{N \in \mathcal{C}^{*}(\mathbb{E})}H(N \cap B)$ 
and~$\bigcap_{N \in \mathcal{C}^{*}(\mathbb{E})}K(N \cap B)$\textup{,} respectively\textup{,} and~at~least one of~the~following statements holds\textup{:}
\begin{list}{}{\topsep=0pt\itemsep=0pt\labelsep=1ex\labelwidth=3ex\leftmargin=\parindent\addtolength{\leftmargin}{\labelsep}\addtolength{\leftmargin}{\labelwidth}}
\item[\textup{(\makebox[1.25ex]{$a$})}]the~subgroups~$H$ and~$K$ coincide and~satisfy a~non-trivi\-al identity\textup{;}
\item[\textup{(\makebox[1.25ex]{$b$})}]the~subgroups~$H$ and~$K$ are~properly contained in~a~subgroup~$D$ of~$B$ satisfying a~non-trivi\-al identity.
\end{list}
If $\mathbb{E}$ is~residually a~$\mathcal{C}$\nobreakdash-group\textup{,} then $\overline{H} = H$ and~$\overline{K} = K$.
\end{eproposition}

\begin{proof}[\textup{\textbf{Proof of~Theorem~\ref{et03}}}]
Suppose that the~subgroups~$\overline{H}$ and~$\overline{K}$ are~defined as~in~Proposition~\ref{ep502}. Since $(H \cap K)\varphi = H \cap K$, the~relations $H \leqslant K$, $H = K$, and~$H \geqslant K$ are~equivalent. Therefore, only two cases are~possible: $H = K$ and~$H \ne HK \ne K$. If~$H = B = K$, then the~residual $\mathcal{C}$\nobreakdash-ness of~$B/H$ and~$B/K$ is~obvious. Hence, we~can further assume that $H \ne B \ne K$.

Due to~Proposition~\ref{ep301}, to~end the~proof it suffices to~show that $H$ and~$K$ are~$\mathcal{C}$\nobreakdash-sep\-a\-ra\-ble in~$B$. If~$\mathfrak{U} = \mathfrak{H}$ or~$\mathfrak{U} = \mathfrak{K}$, then the~group $\mathbb{B} = \langle B * B;\,H = K,\,\varphi \rangle$ satisfies all conditions of~Proposition~\ref{ep501}, which ensures the~required separability.

Suppose that $H$ and~$K$ satisfy a~non-trivi\-al identity. Then the~group~$HK$ has the~same property since it is~an~extension of~$K$ by~a~group isomorphic to~$HK/K$ and~$HK/K \cong H/H \cap K$. Thus, the~conditions of~Proposition~\ref{ep502} hold, and~we get the~equalities $\overline{H} = H$ and~$\overline{K} = K$. It~remains to~note that, if~$N \in \mathcal{C}^{*}(\mathbb{E})$, then $N \cap B \in \mathcal{C}^{*}(B)$ by~Proposition~\ref{ep302}, and~therefore these equalities imply the~$\mathcal{C}$\nobreakdash-sep\-a\-ra\-bil\-ity of~$H$ and~$K$ in~$B$.

Now suppose that $\mathfrak{H}$ satisfies a~non-trivi\-al identity. Then, by~Lemma~2 from~\cite{Shirvani1985AM}, it~satisfies a~non-trivi\-al identity of~the~form
\[
\omega(y,x_{1},x_{2}) = \omega_{0}(x_{1},x_{2})y^{\varepsilon_{1}}\omega_{1}(x_{1},x_{2}) \ldots y^{\varepsilon_{n}}\omega_{n}(x_{1},x_{2}),
\]
where $n \geqslant 1$, $\varepsilon_{1}^{\vphantom{1}},\,\ldots,\,\varepsilon_{n}^{\vphantom{1}} = \pm 1$, and~$\omega_{0}^{\vphantom{1}}(x_{1}^{\vphantom{1}},x_{2}^{\vphantom{1}}),\,\ldots,\,\omega_{n}^{\vphantom{1}}(x_{1}^{\vphantom{1}},x_{2}^{\vphantom{1}}) \in \{x_{1}^{\pm 1},\,x_{2}^{\pm 1},\,(x_{1}^{\vphantom{1}}x_{2}^{-1})^{\pm 1}_{\vphantom{1}}\}$. Let~us assume that there exist elements $u_{1},\,u_{2},\,v \in \mathbb{E}$ with~the~following properties:
\begin{list}{}{\topsep=0pt\itemsep=0pt\labelsep=1ex\labelwidth=3ex\leftmargin=\parindent\addtolength{\leftmargin}{\labelsep}\addtolength{\leftmargin}{\labelwidth}}
\item[$(\mathfrak{i})$\phantom{$\mathfrak{i}$}]the commutator $[\omega(v,u_{1},u_{2}),\,v]$ has a~reduced form of~non-zero length;
\item[$(\mathfrak{ii})$]for each subgroup $N \in \mathcal{C}^{*}(\mathbb{E})$, the~inclusions $u_{1},\,u_{2} \in BN$ and~$v \in HN$ hold.
\end{list}
Then $\mathbb{E}$ is~not residually a~$\mathcal{C}$\nobreakdash-group, contrary to~the~condition of~the~theorem.

Indeed, since the~element $g = [\omega(v,u_{1},u_{2}),\,v]$ has a~reduced form of~non-zero length, it~is~not equal to~$1$. At~the~same time, if~$N \in \mathcal{C}^{*}(\mathbb{E})$, then $g \equiv [\omega(h,b_{1},b_{2}),\,h] \pmod N$ for~some $b_{1},\,b_{2} \in B$ and~$h \in H$. The~restriction to~$H$ of~the~conjugation by~$\omega(h,b_{1},b_{2})$ coincides with~the~element $\omega(\widehat{h},\,\widehat{b_{1}}|_{H},\,\widehat{b_{2}}|_{H})$ of~$\mathfrak{H}$, which is~the~identity mapping because $\mathfrak{H}$ satisfies~$\omega$. Therefore, $[\omega(h,b_{1},b_{2}),\,h] = 1$, $g \equiv 1 \pmod N$, and~$\mathbb{E}$ is~not residually a~$\mathcal{C}$\nobreakdash-group since $N$ is~chosen arbitrarily.

As~above, to~prove the~$\mathcal{C}$\nobreakdash-sep\-a\-ra\-bil\-ity of~$H$ and~$K$ in~$B$, it~suffices to~show that $\overline{H} = H$ and~$\overline{K} = K$. Arguing by~contradiction, we~consider four cases and,~in~each of~them, find elements $u_{1},\,u_{2},\,v \in \mathbb{E}$ satisfying~$(\mathfrak{i})$ and~$(\mathfrak{ii})$.

\smallskip

\textit{Case~1.} $\overline{H} \ne H$ and~$[B:H] \geqslant 3$.

\smallskip

Let $b_{1} \in \overline{H} \setminus H$. Since $[B:H] \geqslant 3$ and~$K \ne B$, there exist elements $b_{2},\,c \in B$ such that $b_{2}^{\vphantom{1}},\,b_{1}^{\vphantom{1}}b_{2}^{-1} \notin H$ and~$c \notin K$. Let~us put $u_{1} = b_{1}$, $u_{2} = b_{2}$, and~$v = tc^{-1}t^{-1}b_{1}tct^{-1}$.

Since $b_{1}^{\pm 1},\,b_{2}^{\pm 1},\,(b_{1}^{\vphantom{1}}b_{2}^{-1})^{\pm 1}_{\vphantom{1}} \notin H$ and~$c \notin K$, the~element $[\omega(v,u_{1},u_{2}),\,v]$ has a~reduced form of~length~$8(n+1)$. At~the~same time, if~$N \in \mathcal{C}^{*}(\mathbb{E})$, then it follows from~the~inclusion $b_{1} \in \overline{H}$ that $b_{1} \in HN$, $t^{-1}b_{1}t \in KN$, $c^{-1}t^{-1}b_{1}tc \in KN$ (because $K$ is~normal in~$B$), and~$tc^{-1}t^{-1}b_{1}tct^{-1} \in HN$. Thus, the~elements~$u_{1}$,~$u_{2}$, and~$v$ satisfy~$(\mathfrak{i})$ and~$(\mathfrak{ii})$.

\smallskip

\textit{Case~2.} $\overline{H} \ne H$ and~$[B:H] = 2$.

\smallskip

Let~us fix some elements $b \in \overline{H} \setminus H$ and~$c \in B \setminus K$, and~put $u_{1} = t^{-1}bt$, $u_{2} = t^{-2}bt^{2}$, and~$v = c$. Then $u_{1}^{\vphantom{1}}u_{2}^{-1} = t^{-1}_{\vphantom{1}}bt^{-1}_{\vphantom{1}}b^{-1}_{\vphantom{1}}t^{2}_{\vphantom{1}}$ and~the~element $[\omega(v,u_{1},u_{2}),\,v]$ has a~reduced form of~length not~less than~$4(n+1)$. The~relations $[B:H] = 2$ and~$\overline{H} \ne H$ mean that $B = \overline{H}$. Therefore, $b \in HN$ and~$BN = HN$ for~any $N \in \mathcal{C}^{*}(\mathbb{E})$. It~follows that
\[
u_{1} = t^{-1}bt \in KN \leqslant BN = HN, \quad 
u_{2} = t^{-2}bt^{2} = t^{-1}u_{1}t \in KN,
\]
and~$v = c \in BN = HN$, as~required.

\smallskip

\textit{Case~3.} $\overline{K} \ne K$ and~$[B:H] \geqslant 3$.

\smallskip

Suppose that $c \in \overline{K} \setminus K$ and~elements $b_{1}^{\vphantom{1}},\,b_{2}^{\vphantom{1}} \in B \setminus H$ are~such that $b_{1}^{\vphantom{1}}b_{2}^{-1} \notin H$. Let~us put $u_{1} = b_{1}$, $u_{2} = b_{2}$, and~$v = tct^{-1}$. Then the~element $[\omega(v,u_{1},u_{2}),\,v]$ has a~reduced form of~length~$4(n+1)$ and,~for~each subgroup $N \in \mathcal{C}^{*}(\mathbb{E})$, the~inclusions $b_{1},\,b_{2} \in BN$, $c \in KN$, and~$tct^{-1} \in HN$ hold.

\smallskip

\textit{Case~4.} $\overline{K} \ne K$ and~$[B:H] = 2$.

\smallskip

It~follows from~the~relations $K/L = K/H \cap K \cong KH/H \leqslant B/H$ that $[K:L] \leqslant 2$. If~$K = L$, then $K \leqslant H$ and,~hence, $K = H$, as~noted at~the~beginning of~the~proof. The~last equality means that $\overline{H} = \overline{K}$ and~$\overline{H} \ne H$, which is~impossible due to~Cases~1 and~2 considered above. Thus, $[K:L] = 2$. Let~us put $\overline{L} = \bigcap_{N \in \mathcal{C}^{*}(\mathbb{E})}L(N \cap B)$ and~show that $\overline{L} \ne L$.

Suppose, on~the~contrary, that $\overline{L} = L$, and~fix some elements $c \in \overline{K} \setminus K$ and~$k \in K \setminus L$. Since $[K:L] = 2$ and~$c \notin K$, the~relations $K = L \cup kL$ and~$c,\,k^{-1}c \notin L = \overline{L}$ hold. Hence, $c \notin L(N_{1} \cap B)$ and~$k^{-1}c \notin L(N_{2} \cap B)$ for~suitable subgroups $N_{1},\,N_{2} \in \mathcal{C}^{*}(\mathbb{E})$. Let~$N = N_{1} \cap N_{2}$. Then $N \in \mathcal{C}^{*}(\mathbb{E})$ by~Proposition~\ref{ep303} and~$c,\,k^{-1}c \notin L(N \cap B)$. It~follows that
\[
c \notin L(N \cap B) \cup kL(N \cap B) = K(N \cap B)
\]
and,~therefore, $c \notin \overline{K}$ despite the~choice of~$c$.

So, $\overline{L} \ne L$. Let~us fix some elements $b \in \overline{L} \setminus L$ and~$d \in B \setminus H$, and~put $u_{1} = tbt^{-1}$, $u_{2} = t^{2}bt^{-2}$, and~$v = dtbt^{-1}d^{-1}$. Then $u_{1}^{\vphantom{1}}u_{2}^{-1} = tbtb^{-1}_{\vphantom{1}}t^{-2}_{\vphantom{1}}$. Since $L \leqslant H$, the~relations $\overline{L} \leqslant \overline{H} = H$ and~$b \in H$ hold. Therefore, $b \notin K$ and~the~element $[\omega(v,u_{1},u_{2}),\,v]$ has a~reduced form of~length not~less than~$8(n+1)$. At~the~same time, since $b \in \overline{L}$ and~$L$ is~normal in~$\mathbb{E}$, the~inclusions $b \in LN$, $tbt^{-1} \in LN$, $t^{2}bt^{-2} \in LN$, and~$dtbt^{-1}d^{-1} \in LN$ hold for~any $N \in \mathcal{C}^{*}(\mathbb{E})$.

\smallskip

Thus, $\overline{H} = H$ and~$\overline{K} = K$. When $\mathfrak{K}$ satisfies a~non-trivi\-al identity, the~proof is~similar.
\end{proof}

\section{Proof of~Theorems~\ref{et04}--\ref{et06} and~Corollary~\ref{ec02}}\label{es06}

If $\mathbb{E} = \langle B,t;\,t^{-1}Ht = K,\,\varphi \rangle$, $Q$~is~a~normal $(H,K,\varphi)$-com\-pat\-i\-ble subgroup of~$B$, and~$\mathcal{C}$ is~a~class of~groups, then~we say that $Q$~is

a)\hspace{1ex}\emph{$\mathcal{C}$\nobreakdash-ad\-mis\-si\-ble} if there exists a~homomorphism of~the~group
\[
\mathbb{E}_{Q} = \langle B/Q,t;\,t^{-1}(HQ/Q)t = KQ/Q,\,\varphi_{Q} \rangle
\]
onto a~group from~$\mathcal{C}$ acting injectively on~$B/Q$;

b)\hspace{1ex}\emph{pre-$\mathcal{C}$-ad\-mis\-si\-ble} if $B/Q \in \mathcal{C}$ and~$HQ/Q \cap KQ/Q = LQ/Q$.

The~next proposition follows from~Theorems~1 and~3 of~\cite{Sokolov2021SMJ1}.

\begin{eproposition}\label{ep601}
Let $\mathbb{E} = \langle B,t;\,t^{-1}Ht = K,\,\varphi \rangle$\textup{,} and~let $\mathcal{C}$ be~a~root class of~groups. Suppose also that $H$ and~$K$ are~$\mathcal{C}$\nobreakdash-sep\-a\-ra\-ble in~$B$ and~each subgroup of~$\mathcal{C}^{*}(B)$ contains a~$\mathcal{C}$\nobreakdash-ad\-mis\-si\-ble subgroup. Then $\mathbb{E}$ is~residually a~$\mathcal{C}$\nobreakdash-group if and~only if $B$ has the~same property.
\end{eproposition}

\begin{eproposition}\label{ep602}
Suppose that the~group $\mathbb{E} = \langle B,t;\,t^{-1}Ht = K,\,\varphi \rangle$ satisfies~$(*)$ and~$\mathcal{C}$~is a~root class of~groups closed under taking quotient groups. Suppose also that $\mathfrak{F} \in \mathcal{C}$ and the~following conditions hold\textup{:}
\begin{list}{}{\topsep=0pt\itemsep=0pt\labelsep=1ex\labelwidth=2.85ex\leftmargin=\parindent\addtolength{\leftmargin}{\labelsep}\addtolength{\leftmargin}{\labelwidth}}
\item[\textup{(\makebox[1.1ex]{\dag})}]at~least one of~the~subgroups~$\mathfrak{H}$ and~$\mathfrak{K}$ is~normal in~$\mathfrak{U}$ or~$\mathfrak{U} \in \mathcal{C}$\textup{;}
\item[\textup{(\makebox[1.1ex]{\ddag})}]at~least one of~the~subgroups~$\mathfrak{L}$ and~$\mathfrak{F}$ is~normal in~$\mathfrak{V}$ or~$\mathfrak{V} \in \mathcal{C}$.
\end{list}
If $B/H$ and~$B/K$ are~residually $\mathcal{C}$\nobreakdash-groups and~each subgroup of~$\mathcal{C}^{*}(B)$ contains a~pre-$\mathcal{C}$-admis\-si\-ble subgroup\textup{,} then $\mathbb{E}$ and~$B$ are~residually $\mathcal{C}$\nobreakdash-groups simultaneously.
\end{eproposition}

\begin{proof}
In~view of~Propositions~\ref{ep301} and~\ref{ep601}, it~suffices to~show that every pre-$\mathcal{C}$-ad\-mis\-si\-ble subgroup~$Q$ is~$\mathcal{C}$\nobreakdash-ad\-mis\-si\-ble.

Indeed, let the~groups~$\mathbb{E}_{Q}$, $\mathfrak{H}_{Q}$, $\mathfrak{K}_{Q}$, $\mathfrak{L}_{Q}$, $\mathfrak{F}_{Q}$, $\mathfrak{U}_{Q}$, and~$\mathfrak{V}_{Q}$ be~defined as~in~Proposition~\ref{ep404}. By~the~latter, there exists a~homomorphism of~$\mathfrak{U}$ onto~the~group~$\mathfrak{U}_{Q}$ mapping~$\mathfrak{H}$ and~$\mathfrak{K}$ onto~$\mathfrak{H}_{Q}$ and~$\mathfrak{K}_{Q}$, respectively. Therefore, if~$\mathfrak{H}$ is~normal in~$\mathfrak{U}$, then $\mathfrak{H}_{Q}$ is~normal in~$\mathfrak{U}_{Q}$ and,~hence, the~latter is~an~extension of~$\mathfrak{H}_{Q}$ by~the~group $\mathfrak{U}_{Q}/\mathfrak{H}_{Q} = \mathfrak{K}_{Q}\mathfrak{H}_{Q}/\mathfrak{H}_{Q} \cong \mathfrak{K}_{Q}/\mathfrak{H}_{Q} \cap \mathfrak{K}_{Q}$. Since $B/Q \in \mathcal{C}$, it~follows from~Proposition~\ref{ep304} that $\operatorname{Aut}_{B/Q}(HQ/Q) \in \mathcal{C}$ and~$\operatorname{Aut}_{B/Q}(KQ/Q) \in \mathcal{C}$. But~$\operatorname{Aut}_{B/Q}(KQ/Q)$ is~isomorphic to~$\mathfrak{K}_{Q}$. Hence, $\mathfrak{H}_{Q},\,\mathfrak{K}_{Q} \in \mathcal{C}$ and~$\mathfrak{U}_{Q} \in \mathcal{C}$ because $\mathcal{C}$ is~closed under taking quotient groups and~extensions. If~$\mathfrak{K}$ is~normal in~$\mathfrak{U}$, then the~relation $\mathfrak{U}_{Q} \in \mathcal{C}$ is~proved in~exactly the~same way, while if $\mathfrak{U} \in \mathcal{C}$, this relation follows from~the~fact that $\mathcal{C}$ is~closed under taking quotient groups.

The~inclusion $\mathfrak{V}_{Q} \in \mathcal{C}$ can be~verified in~a~similar way, the~only difference is~that the~relation $\mathfrak{F}_{Q} \in \mathcal{C}$ is~ensured by~the~condition $\mathfrak{F} \in \mathcal{C}$. Hence, it~follows that $Q$ is~$\mathcal{C}$\nobreakdash-ad\-mis\-si\-ble due~to~Theorem~\ref{et01}.
\end{proof}

\begin{eproposition}\label{ep603}
\textup{\cite[Propositions~5.2,~6.1,~and~6.3, Theorem~2.2]{Sokolov2023JGT}}
If $\mathcal{C}$ is~a~root class of~groups consisting only of~periodic groups\textup{,} then the~following statements hold.

\textup{1.\hspace{1ex}}The~class of~$\mathcal{C}$\nobreakdash-bound\-ed nilpotent groups is~closed under taking subgroups and~quotient groups.

\textup{2.\hspace{1ex}}If the~exponent of~a~$\mathcal{C}$\nobreakdash-bound\-ed nilpotent group is~finite and~is~a~$\mathfrak{P}(\mathcal{C})$\nobreakdash-num\-ber\textup{,} then this group belongs to~$\mathcal{C}$.

\textup{3.\hspace{1ex}}Every $\mathcal{C}$\nobreakdash-bound\-ed nilpotent group is~$\mathcal{C}$\nobreakdash-qua\-si-reg\-u\-lar with~respect to~any of~its subgroups.

\textup{4.\hspace{1ex}}A~subgroup of~a~$\mathcal{C}$\nobreakdash-bound\-ed nilpotent group~$X$ is~$\mathcal{C}$\nobreakdash-sep\-a\-ra\-ble in~this group if and~only if it is~$\mathfrak{P}(\mathcal{C})^{\prime}$\nobreakdash-iso\-lat\-ed in~$X$.
\end{eproposition}

\begin{eproposition}\label{ep604}
If the~group $\mathbb{E} = \langle B,t;\,t^{-1}Ht = K,\,\varphi \rangle$ satisfies~$(*)$ and~$\mathcal{C}$ is~a~root class of~groups\textup{,} then the~following statements hold.

\textup{1.\hspace{1ex}}If $\mathfrak{V} \in \mathcal{C}$\textup{,} then\textup{,} for~each subgroup $R \in \mathcal{C}^{*}(L)$\textup{,} there exists a~subgroup $S \in \mathcal{C}^{*}(L)$ lying in~$R$ and~such that $S\mathfrak{v} = S$ for~any automorphism $\mathfrak{v} \in \mathfrak{V}$.

\textup{2.\hspace{1ex}}If $\mathfrak{U} \in \mathcal{C}$\textup{,} then\textup{,} for~each subgroup $R \in \mathcal{C}^{*}(H)$\textup{,} there exists a~subgroup $S \in \mathcal{C}^{*}(H)$ lying in~$R$ and~such that $S\mathfrak{u} = S$ for~any automorphism $\mathfrak{u} \in \mathfrak{U}$.

\textup{3.\hspace{1ex}}Suppose that $\mathcal{C}$ consists only of~periodic groups and~$N$ is~a~subgroup of~$B$ that is~locally cyclic or~$\mathcal{C}$\nobreakdash-bound\-ed nilpotent. Then\textup{,} for~each subgroup $R \in \mathcal{C}^{*}(N)$\textup{,} there exists a~subgroup $S \in \mathcal{C}^{*}(N)$ lying in~$R$ and~such that $S\mathfrak{a} = S$ for~any automorphism $\mathfrak{a} \in \operatorname{Aut}N$.
\end{eproposition}

\begin{proof}
1.\hspace{1ex}Let $S = \bigcap_{\mathfrak{v} \in \mathfrak{V}}R\mathfrak{v}$. By~Remak's theorem (see., e.g.,~\cite[Theorem~4.3.9]{KargapolovMerzlyakov1982}) the~quotient group~$L/S$ can be~embedded to~the~unrestricted direct product of~the~groups $L/R\mathfrak{v}$, $\mathfrak{v} \in \mathfrak{V}$, each of~which is~isomorphic to~the~$\mathcal{C}$\nobreakdash-group~$L/R$. Therefore, it~follows from~the~condition $\mathfrak{V} \in \mathcal{C}$ and~the~root class definition that $L/S \in \mathcal{C}$. It~is~also clear that $S\mathfrak{v} = S$ for~any $\mathfrak{v} \in \mathfrak{V}$.

2.\hspace{1ex}Let $S = \bigcap_{\mathfrak{u} \in \mathfrak{U}}R\mathfrak{u}$. As~in~the~proof of~Statement~1, it~follows from~the~inclusions $R \in \mathcal{C}^{*}(H)$ and~$\mathfrak{U} \in \mathcal{C}$ that $S \in \mathcal{C}^{*}(H)$. The~equality $S\mathfrak{u} = S$ is~obvious for~any $\mathfrak{u} \in \mathfrak{U}$.

3.\hspace{1ex}By~Proposition~\ref{ep308}, the~exponent~$q$ of~the~$\mathcal{C}$\nobreakdash-group~$N/R$ is~finite. Consider the~subgroup $S = \operatorname{sgp}\{x^{q} \mid x \in N\}$. Clearly, $S \leqslant R$ and~$S\mathfrak{a} = S$ for~any automorphism $\mathfrak{a} \in \operatorname{Aut}N$. It~is~also obvious that $q$ is~a~$\mathfrak{P}(\mathcal{C})$\nobreakdash-num\-ber and~is equal to~the~exponent of~$N/S$. Therefore, if~$N$ is~a~locally cyclic group, then $N/S$ is~a~finite cyclic group, which belongs to~$\mathcal{C}$ by~Proposition~\ref{ep308}. When $N$ is~a~$\mathcal{C}$\nobreakdash-bound\-ed nilpotent group, $N/S$~is~also a~$\mathcal{C}$\nobreakdash-bound\-ed nilpotent group and~$N/S \in \mathcal{C}$ due~to~Statements~1 and~2 of~Proposition~\ref{ep603}.
\end{proof}

\begin{eproposition}\label{ep605}
Suppose that the~group $\mathbb{E} = \langle B,t;\,t^{-1}Ht = K,\,\varphi \rangle$ satisfies~$(*)$\textup{,} $\mathcal{C}$~is a~root class of~groups closed under taking quotient groups\textup{,} $H/L \in \mathcal{C}$\textup{,} and~$\mathfrak{F} \in \mathcal{C}$. Suppose also that at~least one of~the~following statements holds\textup{:}
\begin{list}{}{\topsep=0pt\itemsep=0pt\labelsep=1ex\labelwidth=3ex\leftmargin=\parindent\addtolength{\leftmargin}{\labelsep}\addtolength{\leftmargin}{\labelwidth}}
\item[\textup{(\makebox[1.25ex]{$a$})}]the~group~$B$ is~$\mathcal{C}$\nobreakdash-qua\-si-reg\-u\-lar with~respect to~$HK$\textup{,} \textup{(\makebox[1.1ex]{\dag})}~holds\textup{,} and~$\mathfrak{V} \in \mathcal{C}$\textup{;}

\item[\textup{(\makebox[1.25ex]{$b$})}]the~class~$\mathcal{C}$ consists only of~periodic groups\textup{,} $H$~and~$K$ are~locally cyclic subgroups\textup{,} and~the~group~$B$ is~$\mathcal{C}$\nobreakdash-qua\-si-reg\-u\-lar with~respect to~$HK$\textup{;}

\item[\textup{(\makebox[1.25ex]{$c$})}]the~class~$\mathcal{C}$ consists only of~periodic groups\textup{,} \textup{(\makebox[1.1ex]{\dag})}~and~\textup{(\makebox[1.1ex]{\ddag})} hold\textup{,} and~$B$ is~a~$\mathcal{C}$\nobreakdash-bound\-ed nilpotent group.
\end{list}
If $B/H$ and~$B/K$ are~residually $\mathcal{C}$\nobreakdash-groups\textup{,} then $\mathbb{E}$ and~$B$ are~residually $\mathcal{C}$\nobreakdash-groups simultaneously.
\end{eproposition}

\begin{proof}
First of~all, let~us note that, by~Proposition~\ref{ep603}, a~$\mathcal{C}$\nobreakdash-bound\-ed nilpotent group is~$\mathcal{C}$\nobreakdash-qua\-si-reg\-u\-lar with~respect to~any of~its subgroups. It~is~also known that the~automorphism group of~a~locally cyclic group is~abelian (see, e.g.,~\cite[\S~113, Exercise~4]{Fuchs21973}). Therefore, \textup{(\makebox[1.1ex]{\dag})},~\textup{(\makebox[1.1ex]{\ddag})}, and~the~$\mathcal{C}$\nobreakdash-qua\-si-reg\-u\-lar\-i\-ty of~$B$ with~respect to~$HK$ hold under any of~Statements~\textup{(\makebox[1.25ex]{$a$})}\nobreakdash--\textup{(\makebox[1.25ex]{$c$})}. This fact and~Proposition~\ref{ep602} imply that, to~end the~proof, it~suffices to~fix a~subgroup $M \in \mathcal{C}^{*}(B)$ and~show that it contains a~pre-$\mathcal{C}$-ad\-mis\-si\-ble subgroup.

Let $R = M \cap L$. Then $R \in \mathcal{C}^{*}(L)$ by~Proposition~\ref{ep302}. If~$H$ and~$K$ are~locally cyclic groups, then $L$ is~also locally cyclic. By~Proposition~\ref{ep603}, if~$B$ is~a~$\mathcal{C}$\nobreakdash-bound\-ed nilpotent group, then $L$ has the~same property. Therefore, it~follows from~Statements~1 and~3 of~Proposition~\ref{ep604} that there exists a~subgroup $S \in \mathcal{C}^{*}(L)$ lying in~$R$ and~satisfying the~equality $S\mathfrak{v} = S$ for~any automorphism $\mathfrak{v} \in \mathfrak{V}$. Since $\mathfrak{V} = \operatorname{Aut}_{\mathbb{E}}(L)$, the~subgroup~$S$ turns~out to~be~normal in~$\mathbb{E}$ and,~in~particular, is~$\varphi$\nobreakdash-in\-var\-i\-ant. The~quotient group~$HK/S$ is~an~extension of~the~$\mathcal{C}$\nobreakdash-group~$L/S$ by~a~group isomorphic to~$HK/L$. The~latter, in~turn, is~an~extension of~the~$\mathcal{C}$\nobreakdash-group~$H/L$ by~a~group isomorphic to~$HK/H$. The~equalities $H\varphi =\nolinebreak K$ and~$L\varphi = L$ imply that $K/L \cong H/L$. Since $HK/H \cong K/H \cap K = K/L$ and~the~class~$\mathcal{C}$ is~closed under taking extensions, it~follows that $HK/S \in \mathcal{C}$. The~$\mathcal{C}$\nobreakdash-qua\-si-reg\-u\-lar\-i\-ty of~$B$ with~respect to~$HK$ and~Proposition~\ref{ep307} guarantee the~existence of~a~subgroup $N \in \mathcal{C}^{*}(B)$ such that $N \cap HK = S$. Let~us show that the~subgroup $Q = M \cap N$ is~pre-$\mathcal{C}$-ad\-mis\-si\-ble and,~therefore, is~the~desired one.

Indeed, it~follows from~the~inclusions $M,N\kern-1.5pt{} \in \mathcal{C}^{*}(B)$ and~Proposition~\ref{ep303} that $Q\kern-1pt{} \in\nolinebreak \mathcal{C}^{*}(B)$. The~relations
\[
S \leqslant R = L \cap M = H \cap K \cap M
\]
and~$S\varphi = S$ imply that $Q \cap HK = M \cap (N \cap HK) = S$ and,~therefore,
\[
(Q \cap H)\varphi = (Q \cap H \cap HK)\varphi = S\varphi = S = Q \cap K \cap HK = Q \cap K.
\]
If~$x \in HQ/Q \cap KQ/Q$ and~$x = hQ = kQ$ for~some $h \in H$ and~$k \in K$, then
\[
h^{-1}k \in Q \cap HK = S \leqslant H \cap K.
\]
Hence, $h,\,k \in H \cap K = L$ and~$x \in LQ/Q$. Thus, $HQ/Q \cap KQ/Q \leqslant LQ/Q$ and,~because the~opposite inclusion is~obvious, the~subgroup~$Q$ is~pre-$\mathcal{C}$-ad\-mis\-si\-ble.
\end{proof}

\begin{eproposition}\label{ep606}
Suppose that the~group $\mathbb{E} = \langle B,t;\,t^{-1}Ht = K,\,\varphi \rangle$ satisfies~$(*)$\textup{,} $\mathcal{C}$~is a~root class of~groups closed under taking quotient groups\textup{,} $\mathfrak{F} \in \mathcal{C}$\textup{,} and~there exists a~homomorphism~$\sigma$ of~$B$ onto~a~group from~$\mathcal{C}$ acting injectively on~$L$. Suppose also that at~least one of~the~following statements holds\textup{:}
\begin{list}{}{\topsep=0pt\itemsep=0pt\labelsep=1ex\labelwidth=3ex\leftmargin=\parindent\addtolength{\leftmargin}{\labelsep}\addtolength{\leftmargin}{\labelwidth}}
\item[\textup{(\makebox[1.25ex]{$a$})}]the~group~$B$ is~$\mathcal{C}$\nobreakdash-qua\-si-reg\-u\-lar with~respect to~$HK$\textup{,} \textup{(\makebox[1.1ex]{\ddag})}~holds\textup{,} and~$\mathfrak{U} \in \mathcal{C}$\textup{;}

\item[\textup{(\makebox[1.25ex]{$b$})}]the~class~$\mathcal{C}$ consists only of~periodic groups\textup{,} $H$~and~$K$ are~locally cyclic subgroups\textup{,} and~the~group~$B$ is~$\mathcal{C}$\nobreakdash-qua\-si-reg\-u\-lar with~respect to~$HK$\textup{;}

\item[\textup{(\makebox[1.25ex]{$c$})}]the~class~$\mathcal{C}$ consists only of~periodic groups\textup{,} \textup{(\makebox[1.1ex]{\dag})}~and~\textup{(\makebox[1.1ex]{\ddag})} hold\textup{,} and~$B$ is~a~$\mathcal{C}$\nobreakdash-bound\-ed nilpotent group\textup{;}

\item[\textup{(\makebox[1.25ex]{$d$})}]the~group~$B$ is~$\mathcal{C}$\nobreakdash-qua\-si-reg\-u\-lar with~respect to~$HK$\textup{,} \textup{(\makebox[1.1ex]{\ddag})}~holds\textup{,} and~the~group~$\mathfrak{U}$ coincides with~one of~its subgroups~$\mathfrak{H}$ and~$\mathfrak{K}$.
\end{list}
If $B/H$ and~$B/K$ are~residually $\mathcal{C}$\nobreakdash-groups\textup{,} then $\mathbb{E}$ and~$B$ are~residually $\mathcal{C}$\nobreakdash-groups simultaneously.
\end{eproposition}

\begin{proof}
Replacing, if~necessary, $H$~with~$K$, $\varphi$~with~$\varphi^{-1}$, and~$t$ with~$t^{-1}$, we~can further assume that, if~Statement~\textup{(\makebox[1.25ex]{$d$})} holds, then $\mathfrak{U} = \mathfrak{H}$. As~in~the~proof of~Proposition~\ref{ep605}, \textup{(\makebox[1.1ex]{\dag})},~\textup{(\makebox[1.1ex]{\ddag})},~and~the~$\mathcal{C}$\nobreakdash-qua\-si-reg\-u\-lar\-i\-ty of~$B$ with~respect to~$HK$ hold under any of~Statements~\textup{(\makebox[1.25ex]{$a$})}\nobreakdash--\textup{(\makebox[1.25ex]{$d$})}. Therefore, it~suffices to~show that each subgroup $M \in \mathcal{C}^{*}(B)$ contains some pre-$\mathcal{C}$-ad\-mis\-si\-ble subgroup.

Let $P = M \cap \ker\sigma$ and~$R = (P \cap H) \cap (P \cap K)\varphi^{-1}$. It~follows from~the~inclusions $M,\,\ker\sigma \in \mathcal{C}^{*}(B)$ and~Propositions~\ref{ep302} and~\ref{ep303} that
\[
P \in \mathcal{C}^{*}(B), \quad 
P \cap H \in \mathcal{C}^{*}(H), \quad 
P \cap K \in \mathcal{C}^{*}(K), \quad 
(P \cap K)\varphi^{-1} \in \mathcal{C}^{*}(H),
\]
and~$R \in \mathcal{C}^{*}(H)$. Let~us show that there exists a~subgroup $S \in \mathcal{C}^{*}(H)$ lying in~$R$ and~satisfying the~equality $S\mathfrak{u} = S$ for~any automorphism $\mathfrak{u} \in \mathfrak{U}$.

If $B$ is~a~$\mathcal{C}$\nobreakdash-bound\-ed nilpotent group, then $H$ has the~same property by~Proposition~\ref{ep603}. Therefore, when one of~Statements~\textup{(\makebox[1.25ex]{$a$})}\nobreakdash--\textup{(\makebox[1.25ex]{$c$})} holds, the~desired subgroup~$S$ exists due~to~Statements~2 and~3 of~Proposition~\ref{ep604}. Let Statement~\textup{(\makebox[1.25ex]{$d$})} hold. Since the~class~$\mathcal{C}$ is~closed under taking quotient groups, it~follows from~the~relations
\[
HK/RK \cong H/R(H \cap K) \cong (H/R)/(RL/R)
\]
that $RK \in \mathcal{C}^{*}(HK)$. The~$\mathcal{C}$\nobreakdash-qua\-si-reg\-u\-lar\-i\-ty of~$B$ with~respect to~$HK$ guarantees the~existence of~a~subgroup $T \in \mathcal{C}^{*}(B)$ such that $T \cap HK \leqslant RK$. Let $S = P \cap T \cap H$. Then $S \leqslant RK$ and~$S \in \mathcal{C}^{*}(H)$ due~to~Propositions~\ref{ep302} and~\ref{ep303}. If~$s \in S$, then $s = rk$ for~suitable $r \in R$ and~$k \in K$, and~it follows from~the~inclusions $R,\,S \leqslant P \cap H$ that $k \in P \cap H \cap K \leqslant \ker\sigma \cap L = 1$. Hence, $S \leqslant R$. Since $P$, $T$, and~$H$ are~normal in~$B$, the~subgroup~$S$ has the~same property. Therefore, $S\mathfrak{h} = S$ for~each $\mathfrak{h} \in \mathfrak{H}$, and~$S\mathfrak{u} = S$ for~each $\mathfrak{u} \in \mathfrak{U}$ because $\mathfrak{U} = \mathfrak{H}$.

Thus, a~subgroup~$S$ with~the~required properties always exists. Since $\operatorname{Aut}_{B}(H) = \mathfrak{H} \leqslant \mathfrak{U}$ and~$\varphi\operatorname{Aut}_{B}(K)\varphi^{-1} = \mathfrak{K} \leqslant \mathfrak{U}$, the~equalities $S\mathfrak{h} = S$, $S\varphi \mathfrak{k}\varphi^{-1} = S$, and~$(S\varphi)\mathfrak{k} = S\varphi$ hold for~all $\mathfrak{h} \in \operatorname{Aut}_{B}(H)$ and~$\mathfrak{k} \in \operatorname{Aut}_{B}(K)$. Hence, $S$~and~$S\varphi$ are~normal in~$B$. It~follows from~the~relations
\[
S \leqslant R \leqslant P \cap H, \quad 
S\varphi \leqslant R\varphi \leqslant P \cap K, \quad 
P \cap H \cap K \leqslant \ker\sigma \cap L = 1
\]
that $S \cap K = 1 = S\varphi \cap H$. Therefore, $S {\cdot} S\varphi \cap H = S$ and~$S {\cdot} S\varphi \cap K = S\varphi$.

The~group $HK/S {\cdot} S\varphi$ is~an~extension of~$SK/S {\cdot} S\varphi$ by~a~group isomorphic to~$HK/SK$, and~the~class~$\mathcal{C}$ is~closed under taking extensions and~quotient groups. Therefore, it~follows from~the~relations
\begin{gather*}
SK/S {\cdot} S\varphi \cong K/S\varphi(K \cap S) \cong (K/S\varphi)/(S\varphi(K \cap S)/S\varphi),\\ 
HK/SK \cong H/S(H \cap K) \cong (H/S)/(SL/S),
\end{gather*}
and $K/S\varphi \cong H/S \in \mathcal{C}$ that $HK/S {\cdot} S\varphi \in \mathcal{C}$. Since $S {\cdot} S\varphi$ is~normal in~$B$ and~the~latter is~$\mathcal{C}$\nobreakdash-qua\-si-reg\-u\-lar with~respect to~$HK$, Proposition~\ref{ep307} implies the~existence of~a~subgroup $N \in \mathcal{C}^{*}(B)$ such that $N \cap HK = S {\cdot} S\varphi$. Let $Q = P \cap N$. Then $Q \leqslant P \leqslant M$, and~it follows from~Proposition~\ref{ep303} and~the~inclusions $P,\,N \in \mathcal{C}^{*}(B)$ that $Q \in \mathcal{C}^{*}(B)$. Let~us show that $Q$ is~pre-$\mathcal{C}$-ad\-mis\-si\-ble.

The~relations $S \leqslant P \cap H$ and~$S\varphi \leqslant P \cap K$ imply that
\[
Q \cap HK = P \cap (N \cap HK) = S {\cdot} S\varphi.
\]
Since $S {\cdot} S\varphi \cap H = S$ and~$S {\cdot} S\varphi \cap K = S\varphi$, as~proven above, the~equalities
\[
(Q \cap H)\varphi = (Q \cap H \cap HK)\varphi = S\varphi = Q \cap K \cap HK = Q \cap K
\]
hold. If~$x \in HQ/Q \cap KQ/Q$ and~$x = hQ = kQ$ for~some $h \in H$ and~$k \in K$, then
\[
h^{-1}k \in Q \cap HK = S {\cdot} S\varphi
\]
and~$h^{-1}k = ss^{\prime}$ for~suitable $s \in S$ and~$s^{\prime} \in S\varphi$. Therefore,
\[
hs = k(s^{\prime})^{-1} \in H \cap K = L,
\]
$h \in LS \leqslant LQ$, and~$x \in LQ/Q$. Thus, $HQ/Q \cap KQ/Q = LQ/Q$ and,~hence, $Q$ is~pre-$\mathcal{C}$-admis\-si\-ble.
\end{proof}

Obviously, if~$\mathcal{C}$ is~a~root class of~groups, then the~inclusion $\mathfrak{F} \in \mathcal{C}$ is~guaranteed by~the~condition $\mathfrak{V} \in \mathcal{C}$. Therefore, \textbf{Theorem~\ref{et04}} is~a~special case of~Proposition~\ref{ep607} below, which, in~turn, follows from~Propositions~\ref{ep605} and~\ref{ep606}. To~anticipate possible questions from~the~reader, we~note that Propositions~\ref{ep607} and~\ref{ep608} use Statements~$(\alpha)$ and~$(\beta)$ from~Theorem~\ref{et04}.

\begin{eproposition}\label{ep607}
Suppose that the~group $\mathbb{E} = \langle B,t;\,t^{-1}Ht = K,\,\varphi \rangle$ satisfies~$(*)$ and~$\mathcal{C}$~is a~root class of~groups closed under taking quotient groups. Suppose also that $B$ is~$\mathcal{C}$\nobreakdash-quasi-reg\-u\-lar with~respect to~$HK$ and~at~least one of~the~following statements holds\textup{:}
\begin{list}{}{\topsep=0pt\itemsep=0pt\labelsep=1ex\labelwidth=3ex\leftmargin=\parindent\addtolength{\leftmargin}{\labelsep}\addtolength{\leftmargin}{\labelwidth}}
\item[\textup{(\makebox[1.25ex]{$a$})}]$\mathfrak{V} \in \mathcal{C}$ and~$(\alpha)$ and~\textup{(\makebox[1.1ex]{\dag})} hold\textup{;}
\item[\textup{(\makebox[1.25ex]{$b$})}]$\mathfrak{U},\,\mathfrak{F} \in \mathcal{C}$ and~$(\beta)$ and~\textup{(\makebox[1.1ex]{\ddag})} hold.
\end{list}
If $B/H$ and~$B/K$ are~residually $\mathcal{C}$\nobreakdash-groups\textup{,} then~$\mathbb{E}$ and~$B$ are~residually $\mathcal{C}$\nobreakdash-groups simultaneously.
\end{eproposition}

\noindent
\textbf{Theorem~\ref{et05}} follows from~Propositions~\ref{ep302},~\ref{ep605},~\ref{ep606} and~Theorem~\ref{et03}.

\begin{proof}[\textup{\textbf{Proof of~Theorem~\ref{et06}}}]
1.\hspace{1ex}Since $B/H$ and~$B/K$ are~residually $\mathcal{C}$\nobreakdash-groups, the~subgroups~$H$ and~$K$ are~$\mathcal{C}$\nobreakdash-sep\-a\-ra\-ble in~$B$ by~Proposition~\ref{ep301}. It~easily follows that $L$ is~also $\mathcal{C}$\nobreakdash-sep\-a\-ra\-ble in~$B$ and,~again by~Proposition~\ref{ep301}, $B/L$~is~residually a~$\mathcal{C}$\nobreakdash-group. Hence, if~the~subgroup~$H/L$ is~finite, then it belongs to~$\mathcal{C}$ due~to~Proposition~\ref{ep303}. Conversely, if~the~locally cyclic group~$H/L$ belongs to~$\mathcal{C}$, then it has a~finite exponent by~Proposition~\ref{ep308} and,~therefore, is~finite.

2.\hspace{1ex}If $(\beta)$ holds, then the~locally cyclic group~$L$ can be~embedded in~a~$\mathcal{C}$\nobreakdash-group. As~above, this implies its finiteness. The~opposite statement follows from~Proposition~\ref{ep303}.

3.\hspace{1ex}Necessity is~ensured by~Proposition~\ref{ep302} and~Theorem~\ref{et03}. To~prove sufficiency, let~us show that $B$ is~$\mathcal{C}$\nobreakdash-qua\-si-reg\-u\-lar with~respect to~$HK$. Then the~residual $\mathcal{C}$\nobreakdash-ness of~$\mathbb{E}$ will follow from~Statement~1 of~this theorem and~Proposition~\ref{ep605}.

As~noted in~the~proof of~the~latter, the~quotient group~$HK/L$ is~an~extension of~$H/L$ by~a~group isomorphic to~$H/L$ and,~therefore, is~finite. By~the~arguments used to~verify Statement~1, the~residual $\mathcal{C}$\nobreakdash-ness of~$B/H$ and~$B/K$ implies the~residual $\mathcal{C}$\nobreakdash-ness of~$B/L$. Due~to~Proposition~\ref{ep303}, it~follows that there exists a~subgroup $S/L \in \mathcal{C}^{*}(B/L)$ satisfying the~condition $S/L \cap HK/L = 1$. Clearly, $S \in \mathcal{C}^{*}(B)$ and~$S \cap HK \leqslant L$.

Now, if~$M \in \mathcal{C}^{*}(HK)$ and~$Q = M \cap L$, then $Q \in \mathcal{C}^{*}(L)$ by~Proposition~\ref{ep302}. Since $B$ is~$\mathcal{C}$\nobreakdash-qua\-si-reg\-u\-lar with~respect to~$L$, there exists a~subgroup $R \in \mathcal{C}^{*}(B)$ such that $R \cap L \leqslant Q$. Let $N = R \cap S$. Then $N \in \mathcal{C}^{*}(B)$ by~Proposition~\ref{ep303}~and
\[
N \cap HK = R \cap S \cap HK \leqslant R \cap L \leqslant Q \leqslant M.
\]
Thus, the~group~$B$ is~$\mathcal{C}$\nobreakdash-qua\-si-reg\-u\-lar with~respect to~$HK$, as~required.

4.\hspace{1ex}Sufficiency follows from~Proposition~\ref{ep303}, which ensures that $(\beta)$ holds, and~Proposition~\ref{ep606}. Let~us prove necessity.

\pagebreak

By Proposition~\ref{ep303}, since $\mathbb{E}$ is~residually a~$\mathcal{C}$\nobreakdash-group, it~has a~homomorphism onto~a~group from~$\mathcal{C}$ acting injectively on~the~finite subgroup~$L$. This fact and~Proposition~\ref{ep304} imply that $\mathfrak{V} = \operatorname{Aut}_{\mathbb{E}}(L) \in \mathcal{C}$ and~$\mathfrak{F} \in \mathcal{C}$. As~above, the~residual $\mathcal{C}$\nobreakdash-ness of~the~groups~$B$, $B/H$, and~$B/K$ is~ensured by~Proposition~\ref{ep302} and~Theorem~\ref{et03}.
\end{proof}

\noindent
\textbf{Corollary~\ref{ec02}} can be~deduced either from~Theorems~\ref{et03}\nobreakdash--\ref{et06} and~Propositions~\ref{ep301},~\ref{ep302}, and~\ref{ep603}, or~from~Proposition~\ref{ep608} below. The~second method uses the~fact that the~automorphism group of~a~locally cyclic group is~abelian, which is~already mentioned in~the~proof of~Proposition~\ref{ep605}.

\begin{eproposition}\label{ep608}
Suppose that the~group $\mathbb{E} = \langle B,t;\,t^{-1}Ht = K,\,\varphi \rangle$ satisfies~$(*)$ and~$\mathcal{C}$~is a~root class of~groups consisting only of~periodic groups and~closed under taking quotient groups. Suppose also that $\mathfrak{F} \in \mathcal{C}$ and~$B$ is~a~$\mathcal{C}$\nobreakdash-bound\-ed nilpotent group. Finally\textup{,} let~\textup{(\makebox[1.1ex]{\dag})}\textup{,}~\textup{(\makebox[1.1ex]{\ddag})}\textup{,} and~at~least one of~Statements~$(\alpha)$ and~$(\beta)$ hold. Then $\mathbb{E}$ is~residually a~$\mathcal{C}$\nobreakdash-group if and~only if the~subgroups~$\{1\}$\textup{,} $H$\textup{,} and~$K$ are~$\mathfrak{P}(\mathcal{C})^{\prime}$\nobreakdash-iso\-lat\-ed in~$B$.
\end{eproposition}

\begin{proof}
First of~all, let~us note that, by~Proposition~\ref{ep603}, the~subgroups~$\{1\}$, $H$, and~$K$ are~$\mathfrak{P}(\mathcal{C})^{\prime}$\nobreakdash-iso\-lat\-ed in~$B$ if and~only if they are~$\mathcal{C}$\nobreakdash-sep\-a\-ra\-ble in~this group. Due~to~Proposition~\ref{ep301}, the~latter property is~equivalent to~the~residual $\mathcal{C}$\nobreakdash-ness of~the~groups~$B$, $B/H$, and~$B/K$. Therefore, necessity follows from~Proposition~\ref{ep302} and~Theorem~\ref{et03}, while sufficiency can be~deduced from~Propositions~\ref{ep605} and~\ref{ep606}.
\end{proof}

\end{document}